\documentclass[11pt, twoside]{article}
\pdfoutput=1

\usepackage{graphicx}
\usepackage[caption=false]{subfig}
\usepackage{amsmath}
\usepackage{amssymb,amsfonts}
\usepackage{amsthm}
\usepackage{bm}
\usepackage{mathrsfs}
\usepackage{amssymb}
\usepackage{multirow}

\usepackage[margin=1in]{geometry}

\usepackage{algorithm}
\usepackage{algorithmic}

\numberwithin{equation}{section}
\usepackage{multirow}
\usepackage{makecell}
\usepackage{booktabs}
\theoremstyle{definition}
\newtheorem{theorem}{Theorem}

\newtheorem{lemma}{Lemma}

\newtheorem{remark}{Remark}
\newtheorem{example}{Example}

\usepackage{cite}
\usepackage{hyperref}
\usepackage[nameinlink]{cleveref}

\usepackage{fancyhdr}
\begin{document}
	
	\title{A Superconvergent  Ensemble HDG Method for Parameterized Convection Diffusion Equations}

\author{Gang Chen%
	\thanks{School of Mathematics Sciences, University of Electronic Science and Technology of China, Chengdu, China (\mbox{cglwdm@uestc.edu.cn}).}
	\and
	Liangya Pi%
	\thanks{Department of Mathematics
		and Statistics, Missouri University of Science and Technology, Rolla, MO, USA (\mbox{lpp4f@mst.edu}).}
	\and
	Liwei Xu%
	\thanks{School of Mathematics Sciences, University of Electronic Science and Technology of China, Chengdu, China (\mbox{xul@uestc.edu.cn})}.
	\and
	Yangwen Zhang%
	\thanks{Department of Mathematics Science, University of Delaware, Newark, DE, USA (\mbox{ywzhangf@udel.edu}).}
}

\date{\today}

\maketitle

\begin{abstract}
	In this paper, we first devise an ensemble hybridizable discontinuous Galerkin (HDG) method to efficiently simulate a group of parameterized convection diffusion PDEs. These PDEs have different coefficients, initial conditions, source terms and boundary conditions. The ensemble HDG discrete system shares a common coefficient matrix with multiple right hand side (RHS) vectors; it  reduces both computational cost and storage. We have two contributions in this paper. First, we derive an optimal $L^2$ convergence rate for the ensemble solutions on a general polygonal domain, which is the first such result in the literature. Second, we  obtain a  superconvergent rate for the ensemble solutions  after an element-by-element postprocessing under some assumptions on the domain and  the coefficients of the PDEs.  We present numerical experiments to confirm our theoretical results.
\end{abstract}

\section{Introduction}
A challenge in numerical simulations is to reduce computational cost while keeping accuracy. Toward this end, many fast algorithms have been proposed, which include domain decomposition methods \cite{Quarteroni_DDM_Book_1999}, multigrid methods \cite{Xu_Multigrid_SIAM_Rev_1992},  interpolated coefficient methods  \cite{Douglas_Dupont_Quasi_MathCom_1975,Sanz_Abia_SINUM_1984,CockburnSinglerZhang1}, and so on. These methods are only suitable for a single simulation, not for a group of simulations with different coefficients, initial conditions, source terms and boundary conditions in many scenarios; for example, one needs  repeated simulations to obtain accurate statistical information about the outputs of interest in some uncertainty quantification problems.  A common way is to treat the simulations seperately;  this requires computational effort and memory. Parallel computing is  one method that can solve this problem if sufficient memory is available. 

However, the computational effort and storage requirement is still a great challenge in real simulations. An ensemble method was proposed by Jiang and Layton \cite{Jiang_Layton_Flow_UQ_2014} to address this issue. They studied  a set of $J$ solutions of the Navier-Stokes equations with different  initial conditions and forcing terms. This algorithm uses the mean of the solutions to form a common coefficient  matrix at each time step. Hence,  the problem is reduced to solving one linear system with many right hand side (RHS) vectors, which can be efficiently computed by many existing algorithms, such as LU factorization,  GMRES, etc.  The ensemble scheme has been extended to many  different models; see, e.g.,    \cite{Jiang_Fluid_JSC_2015,Jiang_NS_NUPDE_2017,Jiang_Layton_Fluid_NMPDE_2015,Jiang_Tran_Flow_CMAM_2015,Gunzburger_Jiang_Schneier_NS_SINUM_2017,Gunzburger_Jiang_Schneier_NS_IJNAM_2018,Fiordilino_Bousinesq_SINUM_2018,Gunzburger_Jiang_Wang_Flow_CMAM_2017,Gunzburger_Jiang_Wang_Flow_IMAJNA_2018}. 
Recently, Luo and Wang  \cite{Luo_Wang_Heat_SINUM_2018}  extended this idea to a stochastic parabolic PDE.  It is worthwhile to mention that all the above works only obtained \emph{suboptimal} $L^2$ convergence rate for the ensemble solutions.

All the previous works have used  continuous Galerkin (CG) methods; however,  for high Reynolds number flows \cite{Jiang_Fluid_JSC_2015,Jiang_Layton_Fluid_NMPDE_2015,Takhirov_Neda_Waters__Flow_NMPDE_2016} using a modified CG method may still cause non-physical oscillations.  
The literature on discontinuous Galerkin (DG) methods for simulating a \emph{single} convection diffusion PDE is already substantial and the research in this area is still active; see, e.g.  \cite{Cockburn_Shu_Convection_Diffusion_SINUM_1998,Buffa_Hughes_Convection_Diffusion_SINUM_2006,Cesenek_Jan_Convection_Diffusion_SINUM_2012}. However, there are \emph{no} theoretical or numerical analysis works on  DG methods for the spatial discretization of a group of  parameterized convection diffusion equations.

However, the number of degrees of freedom for DG methods is much larger compared to CG methods; this is  the main drawback of DG methods. Hybridizable discontinuous Galerkin (HDG) methods were originally proposed by Cockburn, Gopalakrishnan, and Lazarov in \cite{Cockburn_Gopalakrishnan_Lazarov_Unify_SINUM_2009} to fix this issue.  The HDG methods are based on a mixed formulation and introduce a numerical flux and a numerical trace to approximate the flux and the trace of the solution.  The discrete HDG global system is only in terms of the numerical trace variable since we can element-by-element eliminate the numerical flux and solution. Therefore, HDG methods have a significantly smaller number of globally coupled degrees of freedom compareed to DG methods.  Moreover,  HDG methods  keep the advantages of DG methods,  which are suitable for convection diffusion problems; see, e.g., \cite{Chen_Cockburn_Convection_Diffusion_IMAJNA_2012,Chen_Cockburn_Convection_Diffusion_MathComp_2014,Fu_Qiu_Zhang_Convection_Dominated_M2AN_2015,Qiu_Shi_Convection_Diffusion_JSC_2016,Chen_Li_Qiu_Posteriori_Convection_Diffusion_IMAJNA_2016}. Also, HDG methods have been applied to flow problems \cite{Cockburn_Shi_Stokes_MathComp_2013,Cockburn_Gopalakrishnan_Nguyen_Peraire_Sayas_Stokes_MathComp_2011,Cockburn_Shi_Stokes_CF_2014,Cockburn_Says_Divergence_Free_MathComp_2014,Cesmelioglu_Cockburn_Nguyen_Peraire_Oseen_JSC_2013,Cockburn_Shi_Stokes_MathComp_2013,Rhebergen_Cockburn_NS_JCP_2013,Rhebergen_Cockburn_Deforming_JCP_2012} and  hyperbolic equations \cite{Cockburn_Nguyen_Peraire_Hyperbolic_Book_2016,Sanchez_Ciuca_Nguyen_Peraire_Cockburn_Hamiltonian_JCP_2017,Stanglmeier_Nguyen_Peraire_Wave_CMAME_2016}.

In this work, we propose a new Ensemble HDG method to investigate a group of parameterized convection diffusion equations on a Lipschitz polyhedral domain  $\Omega\subset \mathbb{R}^{d} $ $ (d\geq 2)$. For $j=1,2,\cdots, J$,  find $(\bm q_j, u_j)$ satisfying
\begin{equation}\label{concection_pde}
\begin{split}
c_j\bm{q}_j +\nabla u_j&= 0 \quad ~ \mbox{in} \; \Omega\times (0,T],\\
\partial_tu_j+{\nabla}\cdot \bm{q}_j
+\bm\beta_j\cdot\nabla u_j&= f_j \quad  \mbox{in} \; \Omega\times (0,T],\\
u_j&=g_j \quad  \mbox{on} \; \partial\Omega\times (0,T],\\
u_j(\cdot,0)&=u_j^0~~~\mbox{in}\ \Omega,
\end{split}
\end{equation}
where the vector vector fields $\bm{\beta}_j$ satisfy
\begin{align}\label{beta_con}
\nabla\cdot\bm{\beta}_j = 0.
\end{align}
We make other smoothness assumptions on the data of system \eqref{concection_pde} for our analysis.

{\bf{The HDG Method.}} To better describe the Ensemble HDG method, we first give the semidiscretization of the system \eqref{concection_pde} use an existing HDG method \cite{Cockburn_Gopalakrishnan_Sayas_Porjection_MathComp_2010}.   Let $\mathcal{T}_h$ be a collection of disjoint  simplexes $K$ that partition $\Omega$ and let  $\partial \mathcal{T}_h$ be the set $\{\partial K: K\in \mathcal{T}_h\}$. Let $e\in \mathcal{E}_h^o $ be the interior face if the Lebesgue measure of  $e = \partial K^+ \cap \partial K^-$ is non-zero, similarly, $e \in \mathcal{E}_h^{\partial}$ be the boundary face if the Lebesgue measure of  $e = \partial K \cap \partial \Omega$ is non-zero. Finally, we  set
\begin{align*}
(w,v)_{\mathcal{T}_h} := \sum_{K\in\mathcal{T}_h} (w,v)_K,   \quad\quad\quad\quad\left\langle \zeta,\rho\right\rangle_{\partial\mathcal{T}_h} := \sum_{K\in\mathcal{T}_h} \left\langle \zeta,\rho\right\rangle_{\partial K},
\end{align*}
where $(\cdot,\cdot)_K$  denotes the $L^2(K)$ inner product and 
$ \langle \cdot, \cdot \rangle_{\partial K} $ denotes the $L^2$ inner product on $\partial K$.

Let $\mathcal{P}^k(K)$ denote the set of polynomials of degree at most $k$ on the element $K$.  We define the following  discontinuous finite element spaces
\begin{align*}
\bm{V}_h  &:= \{\bm{v}\in [L^2(\Omega)]^d: \bm{v}|_{K}\in [\mathcal{P}^k(K)]^d, \forall K\in \mathcal{T}_h\},\\
{W}_h  &:= \{{w}\in L^2(\Omega): {w}|_{K}\in \mathcal{P}^{k }(K), \forall K\in \mathcal{T}_h\},\\
{Z}_h  &:= \{{z}\in L^2(\Omega): {z}|_{K}\in \mathcal{P}^{k+1}(K), \forall K\in \mathcal{T}_h\},\\
{M}_h  &:= \{{\mu}\in L^2(\mathcal{\varepsilon}_h): {\mu}|_{e}\in \mathcal{P}^k(e), \forall e\in \mathcal{E}_h,\mu|_{\mathcal{E}_h^\partial} = 0\}.
\end{align*}

We use the notation  $ \nabla v_h $ and $ \nabla \cdot \bm r_h $ to denote the gradient of $ v_h \in W_h $ and the divergence of $ \bm r_h \in \bm V_h$ applied piecewise  on each element $K\in \mathcal T_h$.

The semidiscrete HDG method finds $(\bm q_{jh},u_{jh},\widehat u_{jh})\in \bm V_h\times W_h\times M_h$ such that for all $j=1,2\cdots, J$ 
\begin{equation}\label{semi_discretization}
\begin{split}
(c_j\bm{q}_{jh},\bm r_h)_{\mathcal T_h} - (u_{jh}, \nabla\cdot\bm r_h)_{\mathcal T_h}  + \langle \widehat u_{jh}, \bm r_h\cdot\bm n\rangle_{\partial \mathcal T_h}&= -\langle g_j, \bm r_h\cdot\bm n\rangle_{\varepsilon_h^\partial},\\
(\partial_tu_{jh}, v_h)_{\mathcal T_h} - (\bm{q}_{jh} + \bm \beta_j u_{jh},\nabla v_h)_{\mathcal T_h} + \langle \widehat {\bm q}_{jh}\cdot \bm n, v_h\rangle_{\partial\mathcal T_h} &\\
+\langle \bm\beta_j\cdot\bm n \widehat u_{jh}, v_h\rangle_{\partial \mathcal T_h} +\langle \bm\beta_j\cdot\bm n g_j, v_h\rangle_{\varepsilon_h^\partial}&= (f_j,v_h)_{\mathcal T_h} ,\\
\langle \widehat {\bm q}_{jh}\cdot \bm n + \bm{\beta}_j\cdot \bm n \widehat u_{jh}, \widehat v_h\rangle_{\partial \mathcal T_h}&=0,
\end{split}
\end{equation}
for all $(\bm r_h,v_h,\widehat v_h)\in \bm V_h\times W_h\times M_h$. Here
the numerical traces on $\partial\mathcal{T}_h$ are defined as
\begin{align}
\widehat{\bm{q}}_{jh}\cdot \bm n &=\bm q_{jh}\cdot\bm n + \tau_j (u_{jh}-\widehat u_{jh})   \qquad ~ \mbox{on} \; \partial \mathcal{T}_h\backslash\varepsilon_h^\partial, \label{HDG_discrete2_h}\\
\widehat{\bm{q}}_{jh}\cdot \bm n &=\bm q_{jh}\cdot\bm n+ \tau_j (u_{jh}-g_j)  \quad~~~~~~ \mbox{on}\;  \varepsilon_h^\partial, \label{HDG_discrete2_i}
\end{align}
where $\tau_j$ are positive stabilization functions defined on $\partial \mathcal T_h$ satisfying
\begin{align*}
\tau_j = \tau + \bm \beta_j\cdot \bm n ~\textup{on} ~\partial\mathcal T_h,
\end{align*}
and the function $\tau$ is a positive  constant on each element $K\in\mathcal T_h$.

{\bf{The Ensemble HDG Method.}} It is obvious to see that the system \eqref{semi_discretization}-\eqref{HDG_discrete2_i} has $J$ different coefficient matrices. The idea of the Ensemble HDG method is to treat the system to share one common coefficient matrix by changing the variables $c_j$ and $\bm \beta_j$ into their ensemble means. Before we define the Ensemble HDG method, we give some notation first. 

Suppose the time domain $[0,T]$ is uniformly partition into $N$ steps with time step  $\Delta t$  and let $t_n = n\Delta t$ for $n=1,2\cdots, N$.  Moreover,  $\bar c^n $ and $ \bar{\bm{\beta}}^n $ stand for the ensemble means of the inverse coefficient of diffusion and convection coefficient at time $t_n$, respectively, defined by 
\begin{align}\label{ensemblemena}
\bar c^n = \frac 1J \sum_{j=1}^J c_j^n \qquad \textup{and} \qquad
\bar{\bm{\beta}}^n = \frac 1J \sum_{j=1}^J \bm{\beta}_j^n,
\end{align}
the superscript $n$ denotes the function value at the time $t_n$.

Substitute \eqref{HDG_discrete2_h}-\eqref{HDG_discrete2_i} into \eqref{semi_discretization},  and use some simple algebraic manipulation, the ensemble mean \eqref{ensemblemena},  and the previous step to replace the current step to  obtain the Ensemble  HDG formulation: find $(\bm q^n_{jh},u^n_{jh},\widehat u^n_{jh})\in \bm V_h\times W_h\times M_h$ such that  for all $j=1,2,\cdots, J$ 
\begin{subequations}\label{or}
	\begin{align}
	(\bar c^n\bm{q}^{n}_{jh},\bm{r}_h)_{\mathcal{T}_h}-(u^n_{jh},\nabla\cdot \bm{r}_h)_{\mathcal{T}_h}+ \langle\widehat{u}^n_{jh},\bm{r}_h\cdot \bm n \rangle_{\partial{\mathcal{T}_h}} =
	((\bar c^n-c_j^n)\bm{q}^{n-1}_{jh},\bm{r}_h)_{\mathcal{T}_h}\nonumber\\
	\quad-\langle g_j^n,\bm{r}_h\cdot \bm n \rangle_{{\mathcal{E}^{\partial}_h}}, \label{full_discretion_ensemble_a}\\
	(\partial^+_tu^n_{jh},v_h)_{\mathcal T_h}+(\nabla\cdot\bm{q}^n_{jh}, v_h)_{\mathcal{T}_h}-\langle\bm{q}^n_{jh}\cdot \bm{n},\widehat{v}_h\rangle_{\partial{\mathcal{T}_h}}+( \overline{\bm \beta}^n\cdot \nabla u_{jh}^{n}, v_h)_{\mathcal{T}_h}\nonumber\\
	-\langle \overline{\bm \beta}^n\cdot\bm n, {u}_{jh}^{n} \widehat{v}_h\rangle_{\partial\mathcal{T}_h}+\langle \tau ( u^n_{jh}-{\widehat{u}}^n_{jh}),v_h-\widehat{v}_h \rangle_{\partial\mathcal{T}_h}=
	(f^n_j,v_h)_{\mathcal{T}_h}
	+\langle \tau g_j^n,v_h\rangle_{\mathcal{E}^{\partial}_h}\nonumber\\
	+( (\overline{\bm \beta}^n-\bm{\beta}^n_j)\cdot \nabla u_{jh}^{n-1}, v_h)_{\mathcal{T}_h} -\langle (\overline{\bm \beta}^n-\bm \beta_j^n)\cdot\bm n, {u}_{jh}^{n-1} \widehat{v}_h\rangle_{\partial\mathcal{T}_h},\label{full_discretion_ensemble_c}
	\end{align}
\end{subequations}
for all $(\bm r_h,v_h,\widehat{v}_h)\in \bm V_h\times W_h\times M_h$. The initial conditions $u_{jh}^0$ and $\bm q_{jh}^0$ will be specified later. Finally, we let
\begin{align*}
\partial^+_tu^n_{jh} = \frac{u^n_{jh} - u^{n-1}_{jh}}{\Delta t}.
\end{align*}

It is easy to see that the system \eqref{or} shares one matrix with $J$ RHS vectors, and it is more efficient to solve than performing $J$ separate simulations.  It is worth  mentioning that  this is the \emph{first} time that an ensemble scheme has been derived incorporating HDG  methods;  it is even the \emph{first} time for DG methods.  We provide  a rigorous error analysis to obtain an optimal $L^2$ convergence rate for the flux $\bm q_j$  and the  solution $u_j$ on general polygonal domain $\Omega$ in \Cref{error_analysis}. To the best of our knowledge, this is the \emph{first} time in the literature.  One of the excellent features of HDG methods is that we can obtain  superconvergence after an element-by-element postprocessing; we show that this result also holds in the Ensemble HDG algorithm under some conditions on the domain $\Omega$ and the velocity vector fields $\bm{\beta}_j$. This is also the \emph{first superconvergent} ensemble algorithm in the literature. Finally, some numerical experiments are presented to confirm our theoretical results in \Cref{Numericalexperiments}. Furthermore,
we also present numerical results for convection dominated problems with $c_j^{-1} \ll 1$ to demonstrate the performance of the Ensemble  HDG method in this difficult case. The results show that the Ensemble HDG method is able to capture sharp layers in
the solution. A thorough error analysis of the Ensemble HDG method for the convection dominated case will be in another paper.

\section{Stability}
\label{Stability}
We begin with some notation. We use the standard notation $W^{m,p}(D)$ for Sobolev spaces on $D$ with norm $\|\cdot\|_{m,p,D}$ and seminorm $|\cdot|_{m,p,D}$.  We also write $H^{m}(D)$ instead of $W^{m,2}(D)$, and we omit the index $p$ in the corresponding norms and seminorms. Also, we omit the index $m$ when $m=0$ in the corresponding norms
and seminorms. Moreover, we drop the subscript $D$ if there is no ambiguity in the statement.  We denote by $C(0,T;W^{m,s}(\Omega))$ the Banach space of all continuous functions from $[0,T]$ into $W^{m,s}(\Omega)$, and $L^{p}(0,T;W^{m,s}(\Omega))$ for $1\le p\le \infty$ is similarly defined.

To obtain the stability of \eqref{or} in this section,  we assume the data of  \eqref{concection_pde} satisfies 
\begin{description}
	\item[\textbf{(A1):}]  $ f_j \in C(0,T; L^2(\Omega)) $,  $g_j\in C(0,T; H^{1/2}(\partial \Omega))$, $u_j^0\in L^2(\Omega)$,  $c_j\in C(0,T; L^\infty(\Omega))$ and the vector fields $\bm{\beta}_j \in C(0,T; W^{1,\infty}(\Omega))$.
	
	\item[\textbf{(A2):}] There exists a postive constant $c_0$ such that $c_j^n \ge c_0 $,  and  the  ensemble mean satisfies the condition
	\begin{align}
	|\bar c^n-c_j^n| <  \min \{\bar c^n, \bar c^{n-1}\},  \ \forall  \bm x\in\overline\Omega  \ \ \textup{and} \ 1\le n \le N, 1\le j\le J. \label{condition_c}
	\end{align}
\end{description}

It is worth mentioning that we don't assume any conditions like \eqref{condition_c} on the functions $\bm \beta_j$. The function  $\tau$ is a piecewise constant function independent of $j$ satisfying
\begin{align}
\min_{1\le j\le J}(\tau+\frac{1}{2}\bm\beta_j\cdot\bm n) \ge  \frac 1 2 \max_{1\le j\le J}\|\bm{\beta}_j\|_{0,\infty}, \forall \bm x\in \partial\mathcal{T}_h.\label{tau}
\end{align}

Next,  let $ \Pi_\ell $ and $ P_M $ denote the standard $L^2$  projection operators $\Pi_{\ell}: L^2(K)\to \mathcal P^{\ell}(K)$ and $P_M : L^2(e)\to \mathcal P^{k}(e)$ satisfying
\begin{subequations}\label{L2}
	\begin{align}
	(\Pi_{\ell} w, v_h)_K &= (w,v_h)_K,\quad \forall v_h\in \mathcal P^{\ell}(K),\label{L2_do}\\
	\langle P_M  w, \widehat v_h \rangle_e &= \langle w, \widehat v_h \rangle_e,\quad \forall  \widehat v_h\in \mathcal P^{k}(e).\label{L2_edge}
	\end{align}
\end{subequations}
The following error estimates for the $ L^2 $ projections are standard:
\begin{lemma}\label{lemmainter}
	Suppose $k, \ell \ge 0$. There exists a constant $C$ independent of $K\in\mathcal T_h$ such that
	\begin{subequations}
		\begin{align}
		&\|w - \Pi_{\ell}  w\|_K \le Ch^{\ell+1} |w|_{\ell+1,K},  &  &\forall w\in H^{\ell+1}(K), \label{lemmainter_orthoo}\\
		&\|w- P_M w\|_{\partial K} \le Ch^{k+1/2} |w|_{k+1,K}  &  &\forall w\in H^{k+1}(K). \label{lemmainter_orthoe}
		\end{align}
	\end{subequations}
\end{lemma}
Moreover, the vector $L^2$ projection $\bm \Pi_{\ell}$ is defined similarly. 

We choose the initial conditions $u_{jh}^0 = \Pi_{k+1} u_0, \bm q_{jh}^0=-\nabla u_{jh}^0/c_j^{0}$. To make the presentation simple for the stability, we assume $g_j=0$ for $j=1,2\cdots, J$ in this section.
\begin{lemma} 
	If condition \eqref{condition_c} holds, then the Ensemble HDG formulation is  unconditionally  stable and we have the following estimate:
	\begin{align*}
	\hspace{0.2em}&\hspace{-0.2em}\max_{1\le n\le N}\|{u}^n_{jh}\|^2_{\mathcal{T}_h}
	+\sum_{n=1}^N\|{u}^n_{jh}-{u}^{n-1}_{jh}\|^2_{\mathcal{T}_h}+\Delta t\sum_{n=1}^N\|\sqrt{\bar c^n}\bm{q}^{n}_{jh}\|^2_{\mathcal{T}_h}+\|\sqrt{\tau}(u^n_{jh}-{\widehat{u}}^n_{jh})\|^2_{\partial\mathcal{T}_h}\nonumber\\
	&\le C\Delta t \sum_{n=1}^N \|f_j^n\|_{\mathcal T_h}^2 + C\|u_{jh}^0\|_{\mathcal T_h}^2 + C\|\bm q_{jh}^0\|_{\mathcal T_h}^2.
	\end{align*}
\end{lemma}
\begin{proof}
	Take $(\bm r_h,v_h,\widehat{v}_h)=(\bm{q}^{n}_{jh},u^n_{jh},{\widehat{u}^n}_{jh} )$ in \eqref{or},  use the polarization identity
	\begin{align}\label{polarization}
	(a-b)a=\frac{1}{2}(a^2-b^2+(a-b)^2),
	\end{align}
	and add the  \Cref{full_discretion_ensemble_a}  and \Cref{full_discretion_ensemble_c} together to give
	\begin{align}\label{or1}
	\begin{split}
	\hspace{0.1em}&\hspace{-0.1em} \frac{\|{u}^n_{jh}\|^2_{\mathcal{T}_h}-\|{u}^{n-1}_{jh}\|^2_{\mathcal{T}_h}}{2\Delta t}
	+\frac{\|{u}^n_{jh}-{u}^{n-1}_{jh}\|^2_{\mathcal{T}_h}}{2\Delta t} +\|\sqrt{\bar c^n}\bm{q}^{n}_{jh}\|^2_{\mathcal{T}_h}
	+\|\sqrt{\tau}(u^n_{jh}- \widehat{u}^n_{jh})\|^2_{\partial\mathcal{T}_h}\\
	&= - ( \overline{\bm \beta}^n\cdot \nabla u^n_{jh}, u^n_{jh})_{\mathcal{T}_h} + \langle (\overline{\bm \beta}^n\cdot\bm n) {u}^n_{jh},  \widehat{u}^n_{jh}\rangle_{\partial\mathcal{T}_h}+
	((\overline c^n-c^n_j)\bm q_{jh}^{n-1}, \bm{q}^{n}_{jh} )_{\mathcal{T}_h}\\
	&\quad+( (\overline{\bm \beta}^n-\bm{\beta}_j^n)\cdot \nabla u_{jh}^{n-1}, u^n_{jh})_{\mathcal{T}_h}-\langle (\overline{\bm \beta}^n-\bm{\beta}_j^n)\cdot\bm n, {u}_{jh}^{n-1} \widehat{u}^n_{jh}\rangle_{\partial\mathcal{T}_h}+ (f_j^n,  u^n_{jh})_{\mathcal{T}_h}.
	\end{split}
	\end{align}
	By Green's formula and the fact
	$\langle (\overline{\bm \beta}^n\cdot\bm n)\widehat u^{n}_{jh},  \widehat{u}^n_{jh}\rangle_{\partial\mathcal{T}_h}=0$,
	we have
	\begin{align*}
	- ( \overline{\bm \beta}^n\cdot \nabla u^n_{jh}, u^n_{jh})_{\mathcal{T}_h} + \langle (\overline{\bm \beta}^n\cdot\bm n) {u}^n_{jh},  \widehat{u}^n_{jh}\rangle_{\partial\mathcal{T}_h}
	\le \frac{1}{2}\|\sqrt{|\overline{\bm \beta}^n\cdot\bm n|}
	(u^n_{jh}-{\widehat{u}^n}_{jh})\|^2_{\partial\mathcal{T}_h}.
	\end{align*}
	Hence,  condition \eqref{tau} gives
	\begin{align*}
	\hspace{0.2em}&\hspace{-0.2em}\frac{\|{u}^n_{jh}\|^2_{\mathcal{T}_h}-\|{u}^{n-1}_{jh}\|^2_{\mathcal{T}_h}}{2\Delta t}
	+\frac{\|{u}^n_{jh}-{u}^{n-1}_{jh}\|^2_{\mathcal{T}_h}}{2\Delta t} +\|\sqrt{\bar c^n}\bm{q}^{n}_{jh}\|^2_{\mathcal{T}_h}
	+\frac 1 2 \|\sqrt{\tau}(u^n_{jh}-{\widehat{u}}^n_{jh})\|^2_{\partial \mathcal{T}_h}\nonumber\\
	&\le 
	((\overline c^n-c^n_j)\bm q_{jh}^{n-1},\bm{q}^{n}_{jh})_{\mathcal{T}_h}+( (\overline{\bm \beta}^n-\bm{\beta}_j^n)\cdot \nabla u_{jh}^{n-1}, u^n_{jh})_{\mathcal{T}_h} \\
	&\quad -\langle (\overline{\bm \beta}^n-\bm{\beta}_j^n)\cdot\bm n,{u}_{jh}^{n-1} \widehat{u}^n_h\rangle_{\partial\mathcal{T}_h} + (f_j^n, {u}_{jh}^{n})_{\mathcal T_h}\\
	&= R_1 + R_2 + R_3 + R_4.
	\end{align*}
	Next,  we estimate $\{R_i\}_{i=1}^4$.  First, by the condition \eqref{condition_c}, there exist $0<\alpha <1$ such that 
	\begin{align*}
	R_1&=((\overline c^n-c^n_j)\bm q_{jh}^{n-1}, \bm {q}^{n}_{jh})_{\mathcal{T}_h} \le  \frac{\alpha}{2} \|\sqrt{\bar c^n} \bm{q}^{n}_{jh}\|^2_{\mathcal{T}_h} + \frac{\alpha}{2} \|\sqrt{\bar c^{n-1}} \bm{q}^{n-1}_{jh}\|^2_{\mathcal{T}_h}.
	\end{align*}
	The term $R_2 + R_3$ needs a detailed argument. For simplicity, let $\bm \gamma = \overline{\bm \beta}^n-\bm{\beta}_j^n$. We have 
	\begin{align*}
	R_2 + R_3
	&=( \bm \gamma \cdot \nabla u_{jh}^{n-1}
	, u^n_{jh})_{\mathcal{T}_h} -\langle \bm \gamma \cdot\bm n, u_{jh}^{n-1}
	\widehat{u}^n_{jh}\rangle_{\partial\mathcal{T}_h}\\
	&=(  (\bm \gamma  - \bm{\Pi}_0 \bm \gamma )\cdot \nabla u_{jh}^{n-1}
	, u^n_{jh})_{\mathcal{T}_h} -\langle (\bm \gamma - \bm{\Pi}_0 \bm \gamma)  \cdot\bm n, u_{jh}^{n-1} \widehat{u}^n_{jh}\rangle_{\partial\mathcal{T}_h}\\
	&\quad + (\bm{\Pi}_0 \bm \gamma  \cdot \nabla u_{jh}^{n-1}
	, u^n_{jh})_{\mathcal{T}_h} - \langle \bm{\Pi}_0 \bm \gamma \cdot\bm n, u_{jh}^{n-1}
	\widehat{u}^n_{jh}\rangle_{\partial\mathcal{T}_h}\\
	&=(  (\bm \gamma  - \bm{\Pi}_0 \bm \gamma )\cdot \nabla u_{jh}^{n-1}, u^n_{jh})_{\mathcal{T}_h} -\langle (\bm \gamma - \bm{\Pi}_0 \bm \gamma)  \cdot\bm n, u_{jh}^{n-1} \widehat{u}^n_{jh}\rangle_{\partial\mathcal{T}_h}\\
	&\quad +(\bar c^n \bm q^n_{jh},\bm{\Pi}_0 \bm \gamma  u_{jh}^{n-1}
	)_{\mathcal{T}_h} -
	( ( \bar c^n-c_j^n)\bm{q}^{n-1}_{jh},\bm{\Pi}_0 \bm \gamma  u_{jh}^{n-1})_{\mathcal{T}_h},
	\end{align*}
	where we used   \Cref{full_discretion_ensemble_a} in the last identity.  Hence, 
	\begin{align*}
	R_2+R_3 &=(  (\bm \gamma  - \bm{\Pi}_0 \bm \gamma )\cdot \nabla u_{jh}^{n-1}
	, u^n_{jh})_{\mathcal{T}_h} -\langle (\bm \gamma - \bm{\Pi}_0 \bm \gamma)  \cdot\bm n, u_{jh}^{n-1}
	\widehat{u}^n_{jh}\rangle_{\partial\mathcal{T}_h}\\
	&\quad + (\bar c^n \bm q^n_{jh},\bm{\Pi}_0 \bm \gamma  u_{jh}^{n-1} )_{\mathcal{T}_h} -
	( ( \bar c^n-c_j^n)\bm{q}^{n-1}_{jh},\bm{\Pi}_0 \bm \gamma  u_{jh}^{n-1}
	)_{\mathcal{T}_h}\\
	&\le \sum_{K\in \mathcal T_h} \|\bm \gamma - \bm\Pi_0 \bm \gamma\|_{\infty, K} \|\nabla u_{jh}^{n-1}\|_K \| u^n_{jh}\|_K\\
	&\quad  +  \sum_{K\in \mathcal T_h} \|\bm \gamma - \bm\Pi_0 \bm \gamma\|_{\infty, \partial K} \|u_{jh}^{n-1}\|_{\partial K} (\| \widehat{u}^n_{jh}  - {u}^n_{jh}\|_{\partial K} + \|{u}^n_{jh}\|_{\partial K})\\
	& \quad + \|\bm \Pi_0 \bm \gamma\|_{\infty,\mathcal T_h} \|\bar c^n \bm q^n_{jh}\|_{\mathcal T_h} \|u_{jh}^{n-1}\|_{\mathcal T_h} \\
	&\quad + \| (\bar c^n-c_j^n)\bm \Pi_0\bm \gamma\|_{\infty,\mathcal T_h} \| \bm{q}^{n-1}_{jh}\|_{\mathcal T_h}  \|u_{jh}^{n-1}\|_{\mathcal T_h}\\
	& = R_{31} + R_{32} + R_{33} + R_{34}.
	\end{align*}
	For $R_{31}$, use the local inverse inequality:
	\begin{align*}
	R_{31}  \le   C  \sum_{K\in \mathcal T_h} h_K \|\bm \gamma \|_{1,\infty, K} h_K^{-1} \|u_{jh}^{n-1}\|_K \| u^n_{jh}\|_K\le C (\|{u}^{n-1}_{jh}\|_{\mathcal T_h}^2  + \|{u}^{n}_{jh}\|_{\mathcal T_h}^2). 
	\end{align*}
	Apply the trace inequality and inverse inequality for the term $R_{32}$ to give
	\begin{align*}
	R_{32}  &\le   C  \sum_{K\in \mathcal T_h} h_K \|\bm \gamma \|_{1,\infty, K} h_K^{-1/2} \|u_{jh}^{n-1}\|_K (\|  \widehat{u}^n_{jh}  - {u}^n_{jh}\|_{\partial K} +h_K^{-1/2} \|{u}^n_{jh}\|_{K})\\
	%
	%
	&\le C (\|{u}^{n-1}_{jh}\|_{\mathcal T_h}^2 + \|{u}^{n}_{jh}\|_{\mathcal T_h}^2) + \frac 1 4  \|\sqrt{\tau} (\widehat{u}^n_{jh}  - {u}^n_{jh})\|_{\partial \mathcal T_h}^2. 
	\end{align*}
	For the terms $R_{33}$ and $R_{34}$, use  Young's inequality to obtain
	\begin{align*}
	R_{33}&\le \frac{1-\alpha}{4} \|\sqrt{\bar c^n}\bm q^n_{jh}\|_{\mathcal T_h}^2 +  C \|{u}^{n-1}_{jh}\|_{\mathcal T_h}^2,\\
	R_{34} &\le \frac {1-\alpha}{4} \| \sqrt{\bar c^{n-1}} {\bm q}^{n-1}_{jh}\|_{\mathcal T_h}^2 +C\| {u}^{n-1}_{jh}\|_{\mathcal T_h}^2. 
	\end{align*}
	The  Cauchy-Schwarz inequality for the term $R_4$ gives
	\begin{align*}
	R_4 = (f_j^n,  u^n_{jh})_{\mathcal{T}_h} \le \frac 12 (\|f_j^n\|_{\mathcal T_h}^2 + \|u^n_{jh}\|_{\mathcal T_h}^2).
	\end{align*}
	
	We add \eqref{or1} from $n=1$ to $n=N$,
	and use the above inequalities  to get
	\begin{align*}
	\hspace{0.2em}&\hspace{-0.2em}\max_{1\le n\le N}\|{u}^n_{jh}\|^2_{\mathcal{T}_h}
	+\sum_{n=1}^N\|{u}^n_{jh}-{u}^{n-1}_{jh}\|^2_{\mathcal{T}_h}+\Delta t\sum_{n=1}^N\|\sqrt{\bar c^n}\bm{q}^{n}_{jh}\|^2_{\mathcal{T}_h}+\|\sqrt{\tau}(u^n_{jh}-{\widehat{u}}^n_{jh})\|^2_{\partial\mathcal{T}_h}\nonumber\\
	&\le  C\Delta t\sum_{n=1}^N \|e^{u^n}_{jh}\|^2_{\mathcal{T}_h} +C\Delta t \sum_{n=1}^N \|f_j^n\|_{\mathcal T_h}^2 + C\|u_{jh}^0\|_{\mathcal T_h}^2 + C\|\bm q_{jh}^0\|_{\mathcal T_h}^2.
	\end{align*}
	Gronwall's inequality applied to the above inequality  gives the desired result.
\end{proof}

\section{Error analysis}
\label{error_analysis}
The strategy of the  error analysis for the  Ensemble HDG method is based on \cite{ChenCockburnSinglerZhang1} and \cite{Chen_Cockburn_Convection_Diffusion_IMAJNA_2012}. First,  we define the HDG projections, and use an energy argument to obtain an optimal convergence rate for the ensemble solutions. Second, we define an HDG elliptic projection as in \cite{ChenCockburnSinglerZhang1}, which is a crucial step to get the superconvergence. Next, we give our main results, and in the end, we provide a rigorous error estimation for our Ensemble HDG method.

Throughout, we assume the data and the solution of \eqref{concection_pde}  are smooth enough, and the initial conditions  $(\bm q_{jh}^0,u_{jh}^0)$ of the  Ensemble HDG system  \eqref{or} are chosen as in \Cref{Stability}.
%

\subsection{HDG projection}
\label{HDGprojection}
For any $t\in[0,T]$, let $ (\bm{\Pi}_{V}^j \bm{q}_j,\Pi_{W}^j u_j)$ be the HDG projection of $(\bm q_j, u_j)$, where
$\bm{\Pi}_{V}^j \bm{q}_j$ and $\Pi_{W}^ju_j$ denote components of the HDG projection of $\bm{q}_j$ and $u_j$ into $\bm{V}_h$ and $W_h$, respectively. On each element $K\in\mathcal T_h$,  $(\bm{\Pi}_{V}^j \bm{q}_j,\Pi_{W}^j u_j)$ satisfy the following equations
\begin{subequations}\label{HDG_projection_operator}
	\begin{align}
	(\bm\Pi_V^j\bm q_j + \bm \beta_j  \Pi_W^j  u_j ,\bm r)_K&=(\bm q_j +  \bm \beta_j   u_j,\bm r)_K, \label{projection_operator_1}\\
	(\Pi_W^ju_j, w)_K&=(u_j, w)_K ,\label{projection_operator_2}\\
	\langle\bm\Pi_V^j\bm q_j\cdot\bm n+\bm\beta_j\cdot\bm n\Pi_W^j u_j+\tau\Pi_W^j u_j,\mu\rangle_{e} &= \langle\bm q_j\cdot\bm n+\bm\beta_j\cdot\bm n u_j+\tau u_j,\mu\rangle_{e}, \label{projection_operator_3}
	\end{align}
	for all $(\bm r, w,\mu)\in [\mathcal P^{k-1}(K)]^d \times  \mathcal P^{k-1}(K) \times \mathcal P^{k}(e) $ and for all faces $e$ of the simplex $K$.  We notice the projections are only determined by \eqref{projection_operator_3} when $k=0$.
\end{subequations}
The proof of the following lemma is similar to a result established in \cite{Chen_Cockburn_Convection_Diffusion_IMAJNA_2012} and hence is omitted.
\begin{lemma}\label{pro_error}
	Suppose the polynomial degree  satisfies $k\geq 0$ and also $\tau>0$. Then the system \eqref{HDG_projection_operator} is uniquely solvable for $\bm{\Pi}_V^j\bm{q}_j$ and $\Pi_W^j u_j$. Furthermore, there is a constant $C$ independent of $K$ and $\tau$ such that for $\ell_{\bm{q}_j},\ell_{u_j}$ in $[0,k]$  
	\begin{subequations}
		\begin{align*}
		\|{\bm{\Pi}_V^j}\bm{q}_j-\bm q_j\|_K &\leq Ch_K^{\ell_{\bm{q}_j}+1}|\bm{q}_j|_{\bm{H}^{\ell_{\bm{q}_j}+1}(K)}+Ch_K^{\ell_{{u}_j}+1} {|u_j|}_{{H}^{\ell_{{u_j}}+1}(K)},\label{Proerr_q}\\
		\|{{\Pi}_W^j}{u}_j-u_j\|_K &\leq Ch_K^{\ell_{{u}_j}+1}|{u}_j|_{{H}^{\ell_{{u}_j}+1}(K)}+C {h_K^{\ell_{{\bm{q}_j}}+1}}  {|\nabla\cdot \bm{q}_j|}_{{H}^{\ell_{\bm{q}_j}}(K)}.
		\end{align*}
	\end{subequations}
\end{lemma}

\subsection{Main results}
We can now state our main result for the Ensemble HDG method.
\begin{theorem}\label{main_theorem}
	Let $(\bm q_j^n, u_j^n)$ and $(\bm q_{jh}^n, u_{jh}^n)$ be the solution of \eqref{concection_pde} at time $t_n$ and \eqref{or}, respectively. If the coefficients $c_j$ satisfy \eqref{condition_c}, then we have 
	\begin{subequations}\label{errror_qu}
		\begin{align}
		\max_{1\le n\le N} \| u^n_j -u_{jh}^n\|_{\mathcal T_h} &\le  C(h^{k+1}+\Delta t),\\
		\sqrt{\Delta t \sum_{n=1}^N \| \bm q^n_j -\bm q_{jh}^n\|_{\mathcal T_h}^2} &\le  C(h^{k+1}+\Delta t).
		\end{align}
	\end{subequations}
	Moreover, if $k\ge 1$,  the elliptic regularity inequality \eqref{Dual_PDE1} holds and the coefficients of the PDEs are independent of time,  then we have 
	\begin{align}\label{eror_ustar}
	\sqrt{\Delta t \sum_{n=1}^N \|  u^n_j - u_{jh}^{n\star}\|_{\mathcal T_h}^2} \le  C(h^{k+2}+\Delta t),
	\end{align}
	where  $u_{jh}^{n\star}$ is the postprocessed approximation defined in  \eqref{post_process_1}.
\end{theorem}

\begin{remark}
	To the best of our knowledge, all  previous works only contain  \emph{suboptimal} $L^2$ convergence rate for the ensemble solutions $u_j$;  our result \eqref{errror_qu} is the {first} time to obtain the \emph{optimal} $L^\infty(0,T;L^2(\Omega))$ convergence rate on a  general polygonal domain $\Omega$. Moreover, if the coefficients of the PDEs are independent of time, then after an element-by-element postprocessing, we  obtain the superconvergent rate \eqref{eror_ustar} under some conditions on the domain; for example,  a convex domain is sufficient. This is also the first such result in the literature. 
\end{remark}

\subsection{Proof of \eqref{errror_qu} in \Cref{main_theorem}}
\label{Proofof16}
\begin{lemma} \label{lemma_error}
	For all $n=1,2,\cdots, N$, we have the following equalities:
	\begin{subequations}
		\begin{align*}
		(c_j^n  \bm{\Pi}_V^j\bm q_j^n,\bm{r}_h)_{\mathcal{T}_h}-(\Pi_W^ju^n_j,\nabla\cdot \bm{r}_h)_{\mathcal{T}_h}+\langle
		P_M u^n_j,\bm{r}_h\cdot \bm n \rangle_{\partial{\mathcal{T}_h}}
		= (c_j^n  (\bm\Pi_V^j\bm q^n_j-\bm q^n_j),\bm{r}_h)_{\mathcal{T}_h},
		\end{align*}
		and
		\begin{align*}
		&\quad  (\nabla\cdot\bm\Pi_V^j \bm q^n_j, v_h)_{\mathcal{T}_h}-\langle \bm\Pi_V^j \bm q^n_j\cdot \bm{n},\widehat{v}_h\rangle_{\partial{\mathcal{T}_h}}
		+\langle \tau (\Pi_W^j u^n_j-P_M u^n_j),v_h-\widehat{v}_h \rangle_{\partial\mathcal{T}_h}\\
		&\qquad +( {\bm \beta}^n_j \cdot \nabla \Pi_W^j u^n_j, v_h)_{\mathcal{T}_h}
		-\langle{\bm \beta}^n_j \cdot\bm n, (\Pi_W^j u^n_j) \widehat{v}_h\rangle_{\partial\mathcal{T}_h}\\
		&=(f_j^n-\partial_t u^n_j,v_h)_{\mathcal{T}_h},
		\end{align*}
		for all $(\bm r_h,v_h,\widehat{v}_h)\in \bm V_h\times W_h\times M_h$ and $j=1,2,\cdots,J$.
	\end{subequations}
\end{lemma}
\begin{proof} 
	By the definitions of $\Pi_W^j$ in \eqref{projection_operator_2},  $P_M$ in \eqref{L2_edge}, and the first equation \eqref{concection_pde}, we get
	\begin{align*}
	\hspace{1em}&\hspace{-1em}
	(c_j^n  \bm{\Pi}_V^j\bm q_j^n,\bm{r}_h)_{\mathcal{T}_h}-(\Pi_W^ju^n_j,\nabla\cdot \bm{r}_h)_{\mathcal{T}_h}+\langle
	P_M u^n_j,\bm{r}_h\cdot \bm n \rangle_{\partial{\mathcal{T}_h}}\\
	&=(c_j^n\bm{q}_j^n,\bm{r}_h)_{\mathcal{T}_h}-(\Pi_W^ju^n_j,\nabla\cdot \bm{r}_h)_{\mathcal{T}_h}+\langle
	P_M u^n_j,\bm{r}_h\cdot \bm n \rangle_{\partial{\mathcal{T}_h}} +  (c_j^n  (\bm\Pi_V^j\bm q^n_j-\bm q^n_j),\bm{r}_h)_{\mathcal{T}_h} \\
	&=(c_j^n\bm{q}_j^n,\bm{r}_h)_{\mathcal{T}_h}-(u^n_j,\nabla\cdot \bm{r}_h)_{\mathcal{T}_h}+\langle  u^n_j,\bm{r}_h\cdot \bm n \rangle_{\partial{\mathcal{T}_h}} +  (c_j^n  (\bm\Pi_V^j\bm q^n_j-\bm q^n_j),\bm{r}_h)_{\mathcal{T}_h} \\
	&=(c_j^n\bm{q}_j^n + \nabla u_j^n,\bm{r}_h)_{\mathcal{T}_h} +  (c_j^n  (\bm\Pi_V^j\bm q^n_j-\bm q^n_j),\bm{r}_h)_{\mathcal{T}_h}\\
	&= (c_j^n  (\bm\Pi_V^j\bm q^n_j-\bm q^n_j),\bm{r}_h)_{\mathcal{T}_h}.
	\end{align*}
	This proves the first identity. 
	
	Next, we prove the second identity. First
	\begin{align*}
	&\quad(\nabla\cdot\bm\Pi_V^j \bm q_j^n, v_h)_{\mathcal{T}_h}-\langle \bm\Pi_V^j \bm q_j^n\cdot \bm{n},\widehat{v}_h\rangle_{\partial{\mathcal{T}_h}}
	+\langle \tau (\Pi_W^j u_j^n-P_M u_j^n),v_h-\widehat{v}_h \rangle_{\partial\mathcal{T}_h}\\
	&\qquad+( {\bm \beta}_j \cdot \nabla \Pi_W^j u_j, v_h)_{\mathcal{T}_h}
	-\langle{\bm \beta}_j^n \cdot\bm n, (\Pi_W^j u_j^n) \widehat{v}_h\rangle_{\partial\mathcal{T}_h}\\
	&=(\nabla\cdot\bm{q}_j^n, v_h)_{\mathcal{T}_h} + (\nabla\cdot(\bm\Pi_V^j \bm q_j^n - \bm q_j^n), v_h)_{\mathcal{T}_h} 
	-\langle \bm\Pi_V^j \bm q_j^n\cdot \bm{n},\widehat{v}_h\rangle_{\partial{\mathcal{T}_h}} \\
	&\quad +\langle \tau (\Pi_W^j u_j^n-P_M u_j^n),v_h-\widehat{v}_h \rangle_{\partial\mathcal{T}_h} +( {\bm \beta}_j^n \cdot \nabla u_j^n, v_h)_{\mathcal{T}_h} \\
	& \quad +( {\bm \beta}_j^n \cdot \nabla (\Pi_W^j u_j^n - u_j^n), v_h)_{\mathcal{T}_h} 	-\langle{\bm \beta}_j^n \cdot\bm n, (\Pi_W^j u_j^n) \widehat{v}_h\rangle_{\partial\mathcal{T}_h}.
	\end{align*}
	By the definition of $\Pi_V^j$ and $\Pi_W^j$ in \eqref{projection_operator_1} and $\nabla\cdot\bm \beta_j^n=0$, we have 
	\begin{align*}
	\hspace{1em}&\hspace{-1em}(\nabla\cdot(\bm\Pi_V^j \bm q_j^n - \bm q_j^n), v_h)_{\mathcal{T}_h} +( {\bm \beta}_j^n \cdot \nabla (\Pi_W^j u_j^n - u_j^n), v_h)_{\mathcal{T}_h}\\
	&=  -(\bm\Pi_V^j \bm q_j^n - \bm q_j^n, \nabla v_h)_{\mathcal{T}_h} + \langle (\bm\Pi_V^j \bm q_j^n - \bm q_j^n)\cdot\bm n, v_h\rangle_{\partial\mathcal T_h} \\
	&\quad -( {\bm \beta}_j^n (\Pi_W^j u_j^n - u_j^n), \nabla v_h)_{\mathcal{T}_h} + \langle (\bm \beta_j^n \cdot \bm n )(\Pi_W^j u_j^n - u_j^n),  v_h\rangle_{\partial\mathcal{T}_h}\\
	&=  \langle (\bm\Pi_V^j \bm q_j^n - \bm q_j^n)\cdot\bm n, v_h\rangle_{\partial\mathcal T_h}  + \langle (\bm \beta_j^n \cdot \bm n )(\Pi_W^j u_j^n - u_j^n),  v_h\rangle_{\partial\mathcal{T}_h}.
	\end{align*}
	Using $(\nabla\cdot\bm{q}_j^n, v_h)_{\mathcal{T}_h} +( {\bm \beta}_j^n \cdot \nabla u_j^n, v_h)_{\mathcal{T}_h}  = (f_j^n-\partial_t u_j^n,v_h)_{\mathcal{T}_h}$ and  \eqref{projection_operator_3}, we have 
	\begin{align*}
	&\quad(\nabla\cdot\bm\Pi_V^j \bm q_j^n, v_h)_{\mathcal{T}_h}-\langle \bm\Pi_V^j \bm q_j^n\cdot \bm{n},\widehat{v}_h\rangle_{\partial{\mathcal{T}_h}}
	+\langle \tau (\Pi_W^j u_j^n-P_M u_j^n),v_h-\widehat{v}_h \rangle_{\partial\mathcal{T}_h}\\
	&\qquad+( {\bm \beta}_j \cdot \nabla \Pi_W^j u_j, v_h)_{\mathcal{T}_h}
	-\langle{\bm \beta}_j^n \cdot\bm n, (\Pi_W^j u_j^n) \widehat{v}_h\rangle_{\partial\mathcal{T}_h}\\
	&=(f_j^n-\partial_t u_j^n,v_h)_{\mathcal{T}_h} + \langle(\bm\Pi_V^j \bm q_j^n - \bm q_j^n)\cdot \bm{n},v_h-\widehat{v}_h\rangle_{\partial{\mathcal{T}_h}}\\
	&\quad+ \langle \tau (\Pi_W^j u_j^n-P_M u_j^n),v_h-\widehat{v}_h \rangle_{\partial\mathcal{T}_h}  +\langle{(\bm \beta}_j^n \cdot\bm n) (\Pi_W^j u_j^n - u_j^n),  v_h-\widehat{v}_h\rangle_{\partial\mathcal{T}_h}\nonumber\\
	&=(f_j^n-\partial_t u_j^n,v_h)_{\mathcal{T}_h}.
	\end{align*} 
\end{proof}

Then, substracting the result of  \Cref{lemma_error} from the Ensemble HDG system \eqref{or} gives the following error equations.
\begin{lemma} 
	For $\eta^{u^n}_{jh}=u_{jh}^n-\Pi_W^ju^n_j$, $\eta^{\bm q^n}_{jh}=\bm q_{jh}^n-\bm{\Pi}_V^j\bm q^n_j$ and  $ \eta^{\widehat{u}^n}_{jh}=
	\widehat{u}_{jh}^n-	P_Mu_j^n$, for all $j = 1,2,\cdots,J$, we have  the following  error equations:
	\begin{subequations}\label{error0}
		\begin{align}\label{error_aa}
		\begin{split}
		\hspace{1em}&\hspace{-1em}(\bar c^n\eta_{jh}^{\bm q^n},\bm{r}_h)_{\mathcal{T}_h}
		-(\eta^{u^n}_{jh},\nabla\cdot \bm{r}_h)_{\mathcal{T}_h}+\langle \eta^{\widehat{u}^n}_{jh},\bm{r}_h\cdot \bm n \rangle_{\partial{\mathcal{T}_h}} \\
		&=
		( ( \bar c^n-c_j^n)(\bm{q}^{n-1}_{jh}-
		\bm{\Pi}_V^j{\bm{q}}^{n}_{j}),\bm{r}_h)_{\mathcal{T}_h} - (c_j^n  (\bm\Pi_V^j\bm q^n_j-\bm q^n_j),\bm{r}_h)_{\mathcal{T}_h},
		\end{split}
		\end{align}
		and 		
		\begin{align}
		\begin{split}
		\hspace{1em}&\quad (\partial^+_t\eta^{u^n}_{jh},v_h)_{\mathcal T_h}+(\nabla\cdot \eta_{jh}^{\bm q^n}, v_h)_{\mathcal{T}_h}-\langle \eta^{\bm{q}^n}_{jh}\cdot \bm{n},\widehat{v}_h\rangle_{\partial{\mathcal{T}_h}}
		+( \overline{\bm \beta}^n\cdot \nabla \eta^{u^n}_{jh}, v_h)_{\mathcal{T}_h} \\
		& \quad -\langle \overline{\bm \beta}^n\cdot\bm n, \eta^{{u}^n}_{jh} \widehat{v}_h\rangle_{\partial\mathcal{T}_h}+
		\langle \tau ( \eta^{u^n}_{jh}-\eta^{\widehat{u}}_{jh}),v_h-\widehat{v}_h \rangle_{\partial\mathcal{T}_h} \\
		& = (\partial_t{u}_{j}^n-\partial_t^+\Pi_W^j u_j^n, v_h)_{\mathcal{T}_h} + ( (\overline{\bm \beta}^n-\bm{\beta}_j^n)\cdot \nabla (u_{jh}^{n-1}
		-\Pi_W^j{u}_{j}^{n}), v_h)_{\mathcal{T}_h} \\
		&\quad -\langle (\overline{\bm \beta}^n_j-\bm{\beta}_j^n)\cdot\bm n, ({u}_{jh}^{n-1}-\Pi_W^j{u}^n_{j}) \widehat{v}_h\rangle_{\partial\mathcal{T}_h},\label{error_bb}
		\end{split}
		\end{align}
	\end{subequations}
	for all $(\bm r_h,v_h,\widehat{v}_h)\in \bm V_h\times W_h\times M_h$ and $n=1,2,\cdots,N$.
\end{lemma}

\begin{lemma}\label{energy_quu}
	If condition \eqref{condition_c} holds, then we have the following error estimate:
	\begin{equation}
	\begin{split}
	\max_{1\le n\le N}\|\eta^{{u}^n}_{jh}\|_{\mathcal{T}_h}
	+\sqrt{\Delta t\sum_{n=1}^N \|\sqrt{\overline{c}^n}\eta^{\bm{q}^{n}}_{jh}\|^2_{\mathcal{T}_h}} 
	\le C\left(h^{k+1}+\Delta t\right).
	\end{split}
	\end{equation}
\end{lemma}
\begin{proof}
	We take $(\bm r_h,v_h,\widehat{v}_h)=(\eta^{\bm{q}^{n}}_{jh},\eta^{u^n}_{jh},\eta^{{\widehat{u}}^n}_{jh} )$
	in \eqref{error0},  use the identity \eqref{polarization} and add \Cref{error_aa} and \Cref{29} together to get
	\begin{align}
	\hspace{0.11em}&\hspace{-0.1em} \frac{\|\eta^{{u}^n}_{jh}\|^2_{\mathcal{T}_h}-\|\eta^{{u}^{n-1}}_{jh}\|^2_{\mathcal{T}_h}}{2\Delta t}
	+\frac{\|\eta^{{u}^n}_{jh}-\eta^{{u}^{n-1}}_{jh}\|^2_{\mathcal{T}_h}}{2\Delta t} +\|\sqrt{\bar c^n}\eta^{\bm{q}^{n}}_{jh}\|^2_{\mathcal{T}_h}
	+\|\sqrt{\tau}(\eta^{u^n}_{jh}-\eta^{{\widehat{u}}^n}_{jh})\|^2_{\partial \mathcal{T}_h}\nonumber\\
	&= - ( \overline{\bm \beta}^n\cdot \nabla \eta^{u^n}_{jh}, \eta^{u^n}_{jh})_{\mathcal{T}_h} + \langle \overline{\bm \beta}^n\cdot\bm n, \eta^{{u}^n}_{jh} \eta^{\widehat{u}^n}_{jh}\rangle_{\partial\mathcal{T}_h}+
	((\overline c^n-c^n_j)(\bm q_{jh}^{n-1}-\bm{\Pi}_V^j{\bm q}_{j}^n),\eta^{\bm{q}^{n}}_{jh})_{\mathcal{T}_h}\nonumber\\
	&\quad+
	(\partial_t{u}_{j}^n-\partial^+_t\Pi_W^j u_j^n,\eta^{u^n}_{jh})_{\mathcal{T}_h}+( (\overline{\bm \beta}^n-\bm{\beta}_j^n)\cdot \nabla (u_{jh}^{n-1}
	-\Pi_W^j{u}_{j}^{n}), \eta^{u^n}_{jh})_{\mathcal{T}_h}\nonumber\\
	&\quad-\langle (\overline{\bm \beta}^n-\bm{\beta}_j^n)\cdot\bm n, ({u}_{jh}^{n-1}-\Pi_W^j{u}^n_{j}) \eta^{\widehat{u}^n}_{jh}\rangle_{\partial\mathcal{T}_h}.\label{290}
	\end{align}
	By Green's formula and the fact
	$\langle (\overline{\bm \beta}^n\cdot\bm n)\eta^{\widehat{u}^n}_{jh},  \eta^{\widehat{u}^n}_{jh}\rangle_{\partial\mathcal{T}_h}=0$,
	we have
	\begin{align*}
	( \overline{\bm \beta}^n\cdot \nabla \eta^{u^n}_{jh}, \eta^{u^n}_{jh})_{\mathcal{T}_h}
	-\langle \overline{\bm \beta}^n\cdot\bm n, \eta^{{u}^n}_{jh} \eta^{\widehat{u}^n}_{jh}\rangle_{\partial\mathcal{T}_h}
	\le \frac{1}{2}\|\sqrt{|\overline{\bm \beta}^n\cdot\bm n|}
	(\eta^{u^n}_{jh}-\eta^{\widehat{u}^n}_{jh})\|^2_{\partial\mathcal{T}_h}.
	\end{align*}
	Condition \eqref{tau} and  equality \eqref{290} give
	\begin{align}\label{eror_equa_1_NN}
	\begin{split}
	\hspace{0.2em}&\hspace{-0.2em}\frac{\|\eta^{{u}^n}_{jh}\|^2_{\mathcal{T}_h}-\|\eta^{{u}^{n-1}}_{jh}\|^2_{\mathcal{T}_h}}{2\Delta t}
	+\frac{\|\eta^{{u}^n}_{jh}-\eta^{{u}^{n-1}}_{jh}\|^2_{\mathcal{T}_h}}{2\Delta t} +\|\sqrt{\bar c^n}\eta^{\bm{q}^{n}}_{jh}\|^2_{ \mathcal{T}_h}
	+\frac 1 2\|\sqrt{\tau}(\eta^{u^n}_{jh}-\eta^{{\widehat{u}}^n}_{jh})\|^2_{\partial\mathcal{T}_h}\\
	& \le 
	((\overline c^n-c^n_j)(\bm q_{jh}^{n-1}-\bm{\Pi}_V^j{\bm q}_{j}^n),\eta^{\bm{q}^{n}}_{jh})_{\mathcal{T}_h}+
	(\partial_t{u}_{j}^n-\partial^+_t\Pi_W^j u_j^n,\eta^{u^n}_{jh})_{\mathcal{T}_h}\\
	&\quad +\left[( (\overline{\bm \beta}^n-\bm{\beta}_j^n)\cdot \nabla (u_{jh}^{n-1}
	-\Pi_W^j{u}_{j}^{n}), \eta^{u^n}_{jh})_{\mathcal{T}_h} \right.\\
	&\qquad \left.-\langle (\overline{\bm \beta}^n-\bm{\beta}_j^n)\cdot\bm n, ({u}_{jh}^{n-1}-\Pi_W^j{u}^n_{j}) \eta^{\widehat{u}^n}_{jh}\rangle_{\partial\mathcal{T}_h}\right]\\
	&\quad -(c_j^n  (\bm\Pi_V^j\bm q^n_j-\bm q^n_j),\eta_{jh}^{\bm q^n})_{\mathcal{T}_h}\\
	&= R_1 + R_2 + R_3+R_4.
	\end{split}
	\end{align}
	Next,  we estimate $\{R_i\}_{i=1}^4$.  By the condition \eqref{condition_c}, there exist $0<\alpha <1$ such that 
	\begin{align*}
	R_1&=((\overline c^n-c^n_j)(\eta_{jh}^{\bm q^{n-1}}
	-\Delta t\partial_t^+\bm{\Pi}_V^j\bm q_j^n),\eta^{\bm{q}^{n}}_{jh})_{\mathcal{T}_h}
	\nonumber\\
	&\le \frac{\alpha}{2}\left(\|\sqrt{\bar c^n}\eta^{\bm{q}^{n}}_{jh}\|^2_{\mathcal{T}_h}
	+\|\sqrt{\bar c^{n-1}}\eta^{\bm{q}^{n-1}}_{jh}\|^2_{\mathcal{T}_h}
	\right)+C\Delta t^2\|\partial_t^+\bm{\Pi}_V^j\bm q_j^n\|^2_{\mathcal{T}_h},\nonumber\\
	R_2&=(\partial^+_t({u}_{j}^n-\Pi_W^j u^n_j)
	-\partial^+_tu^n_j+\partial_t u_j^n,\eta^{u^n}_{jh})_{\mathcal{T}_h}\nonumber\\
	&\le C\left( \|\partial^+_t({u}_{j}^n-\Pi_W^ju^n_j)\|^2_{\mathcal{T}_h}
	+\|\partial^+_t u^n_j-\partial_t u_j^n\|^2_{\mathcal{T}_h}
	+\|\eta^{u^n}_{jh}\|^2_{\mathcal{T}_h} \right),\\
	R_4&\le \frac{1-\alpha}{8} \|\sqrt{\bar c^n}\eta_{jh}^{\bm q^n}\|_{\mathcal T_h}^2+ Ch^{2k+2}( |u_j^n|^2_{k+1}+|\bm q_j^n|^2_{k+1}).
	\end{align*}
	If we directly estimate $R_3$, we will obtain only suboptimal convergence rates. Therefore, we need a refined analysis for this term. For simplicity, let $\bm \gamma = \overline{\bm \beta}^n-\bm{\beta}_j^n$.  The following argument is  similar to  the proof of the stability \Cref{Stability};  to make the proof self-contained, we  include these details here. First
	\begin{align*}
	R_3
	&=( \bm \gamma \cdot \nabla (u_{jh}^{n-1}
	-\Pi_W^j{u}_{j}^n), \eta^{u^n}_{jh})_{\mathcal{T}_h} -\langle \bm \gamma \cdot\bm n, (u_{jh}^{n-1}
	-\Pi_W^j{u}_{j}^n) \eta^{\widehat{u}^n}_{jh}\rangle_{\partial\mathcal{T}_h}\\
	&=(  (\bm \gamma  - \bm{\Pi}_0 \bm \gamma )\cdot \nabla (u_{jh}^{n-1}
	-\Pi_W^j{u}_{j}^n), \eta^{u^n}_{jh})_{\mathcal{T}_h} \nonumber\\
	&\quad-\langle (\bm \gamma - \bm{\Pi}_0 \bm \gamma)  \cdot\bm n, (u_{jh}^{n-1}
	-\Pi_W^j{u}_{j}^n) \eta^{\widehat{u}^n}_{jh}\rangle_{\partial\mathcal{T}_h}\\
	&\quad + (\bm{\Pi}_0 \bm \gamma  \cdot \nabla (u_{jh}^{n-1}
	-\Pi_W^j{u}_{j}^n), \eta^{u^n}_{jh})_{\mathcal{T}_h} - \langle \bm{\Pi}_0 \bm \gamma \cdot\bm n, (u_{jh}^{n-1}
	-\Pi_W^j{u}_{j}^n) \eta^{\widehat{u}^n}_{jh}\rangle_{\partial\mathcal{T}_h}.
	\end{align*}
	By the error equation \eqref{error_aa}, we have  
	\begin{align*}
	&(\bm{\Pi}_0 \bm \gamma  \cdot \nabla (u_{jh}^{n-1}
	-\Pi_W^j{u}_{j}^n), \eta^{u^n}_{jh})_{\mathcal{T}_h} - \langle \bm{\Pi}_0 \bm \gamma \cdot\bm n, (u_{jh}^{n-1}
	-\Pi_W^j{u}_{j}^n) \eta^{\widehat{u}^n}_{jh}\rangle_{\partial\mathcal{T}_h} \\
	=& \  (  \nabla \cdot [\bm{\Pi}_0 \bm \gamma  (u_{jh}^{n-1}
	-\Pi_W^j{u}_{j}^n)], \eta^{u^n}_{jh})_{\mathcal{T}_h} - \langle [(\bm{\Pi}_0 \bm \gamma \cdot\bm n) (u_{jh}^{n-1}
	-\Pi_W^j{u}_{j}^n)], \eta^{\widehat{u}^n}_{jh}\rangle_{\partial\mathcal{T}_h}\\
	=& \ 	(\bar c^n \eta_{jh}^{\bm q^n},[\bm{\Pi}_0 \bm \gamma  (u_{jh}^{n-1}
	-\Pi_W^j{u}_{j}^n)])_{\mathcal{T}_h}
	+(c_j^n  (\bm\Pi_V^j\bm q^n_j-\bm q^n_j),[\bm{\Pi}_0 \bm \gamma  (u_{jh}^{n-1}
	-\Pi_W^j{u}_{j}^n)])_{\mathcal{T}_h}\nonumber\\
	& -
	( ( \bar c^n-c_j^n)(\bm{q}^{n-1}_{jh}-
	\bm{\Pi}_V^j\bm q_j^n),[\bm{\Pi}_0 \bm \gamma  (u_{jh}^{n-1}
	-\Pi_W^j{u}_{j}^n)])_{\mathcal{T}_h}.
	\end{align*}
	This gives
	\begin{align*}
	R_3  &=(  (\bm \gamma  - \bm{\Pi}_0 \bm \gamma )\cdot \nabla (u_{jh}^{n-1}
	-\Pi_W^j{u}_{j}^n), \eta^{u^n}_{jh})_{\mathcal{T}_h}\nonumber\\
	&\quad -\langle (\bm \gamma - \bm{\Pi}_0 \bm \gamma)  \cdot\bm n, (u_{jh}^{n-1}
	-\Pi_W^j{u}_{j}^n) \eta^{\widehat{u}^n}_{jh}\rangle_{\partial\mathcal{T}_h}\\
	&\quad + (\bar c^n \eta_{jh}^{\bm q^n},\bm{\Pi}_0 \bm \gamma  (u_{jh}^{n-1}
	-\Pi_W^j{u}_{j}^n))_{\mathcal{T}_h}
	+(c_j^n  (\bm\Pi_V^j\bm q^n_j-\bm q^n_j),[\bm{\Pi}_0 \bm \gamma  (u_{jh}^{n-1}
	-\Pi_W^j{u}_{j}^n)])_{\mathcal{T}_h}
	\\
	&\quad -
	( ( \bar c^n-c_j^n)(\bm{q}^{n-1}_{jh}-
	\bm{\Pi}_V^j\bm q_j^n),\bm{\Pi}_0 \bm \gamma  (u_{jh}^{n-1}
	-\Pi_W^j{u}_{j}^n))_{\mathcal{T}_h}.
	\end{align*}
	Hence, 
	\begin{align*}
	R_3&\le \sum_{K\in \mathcal T_h} \|\bm \gamma - \bm\Pi_0 \bm \gamma\|_{\infty, K} \|\nabla (u_{jh}^{n-1}
	-\Pi_W^j{u}_{j}^n)\|_K \| \eta^{u^n}_{jh}\|_K\\
	&\quad  +  \sum_{K\in \mathcal T_h} \|\bm \gamma - \bm\Pi_0 \bm \gamma\|_{\infty, \partial K} \|u_{jh}^{n-1}
	-\Pi_W^j{u}_{j}^n\|_{\partial K} (\|  \eta^{\widehat{u}^n}_{jh}  - \eta^{{u}^n}_{jh}\|_{\partial K} + \|\eta^{{u}^n}_{jh}\|_{\partial K})\\
	& \quad + \|\bm \Pi_0 \bm \gamma\|_{\infty,\mathcal T_h} \|\bar c^n \eta_{jh}^{\bm q^n}\|_{\mathcal T_h} \|u_{jh}^{n-1}
	-\Pi_W^j{u}_{j}^n\|_{\mathcal T_h}\\
	&\quad + \| (\bar c^n-c_j^n)\bm \Pi_0\bm \gamma\|_{\infty,\mathcal T_h} \| \bm{q}^{n-1}_{jh}-
	\bm{\Pi}_V^j\bm q_j^n\|_{\mathcal T_h}  \|u_{jh}^{n-1}
	-\Pi_W^j{u}_{j}^n\|_{\mathcal T_h}\\
	&\quad + \| c_j^n\bm \Pi_0\bm \gamma\|_{\infty,\mathcal T_h} \| \
	\bm{\Pi}_V^j\bm q_j^n-\bm q_j^n\|_{\mathcal T_h}  \|u_{jh}^{n-1}
	-\Pi_W^j{u}_{j}^n\|_{\mathcal T_h}\\
	& = R_{31} + R_{32} + R_{33} + R_{34}
	+R_{35}.
	\end{align*}
	For $R_{31}$, use the local inverse inequality:
	\begin{align*}
	R_{31}  &\le   C  \sum_{K\in \mathcal T_h} h_K \|\bm \gamma \|_{1,\infty, K} h_K^{-1} \|u_{jh}^{n-1}
	-\Pi_W^ju_j^n\|_K \| \eta^{u^n}_{jh}\|_K\\
	& \le C  \sum_{K\in \mathcal T_h}   \|u_{jh}^{n-1}
	-\Pi_W^ju_j^n\|_K \| \eta^{u^n}_{jh}\|_K\\
	&\le C (\|\eta^{{u}^{n-1}}_{jh}\|_{\mathcal T_h}^2 + \Delta t^2 \| \partial_t^+\Pi_W^ju_j^n\|_{\mathcal T_h}^2 + \|\eta^{{u}^{n}}_{jh}\|_{\mathcal T_h}^2). 
	\end{align*}
	Apply the trace inequality and inverse inequality for the term $R_{32}$ to give
	\begin{align*}
	R_{32}  &\le   C  \sum_{K\in \mathcal T_h} h_K \|\bm \gamma \|_{1,\infty, K} h_K^{-1/2} \|u_{jh}^{n-1}
	-\Pi_W^ju_j^n\|_K (\|  \eta^{\widehat{u}^n}_{jh}  - \eta^{{u}^n}_{jh}\|_{\partial K} +h_K^{-1/2} \|\eta^{{u}^n}_{jh}\|_{K})\\
	& \le C \sum_{K\in \mathcal T_h}  \|u_{jh}^{n-1}
	-\Pi_W^ju_j^n\|_K (\| \eta^{\widehat{u}^n}_{jh}  - \eta^{{u}^n}_{jh}\|_{\partial K} +\|\eta^{{u}^n}_{jh}\|_{K})\\
	&\le C (\|\eta^{{u}^{n-1}}_{jh}\|_{\mathcal T_h}^2 + \Delta t^2 \| \partial_t^+\Pi_W^ju_j^n\|_{\mathcal T_h}^2 + \|\eta^{{u}^{n}}_{jh}\|_{\mathcal T_h}^2) + \frac 1 4  \| \sqrt{\tau}(\eta^{\widehat{u}^n}_{jh}  - \eta^{{u}^n}_{jh})\|_{\partial \mathcal T_h}^2. 
	\end{align*}
	For the terms $R_{33}$, $R_{34}$, $R_{35}$ and $R_4$, use  Young's inequality to obtain
	\begin{align*}
	R_{33}&\le \frac{1-\alpha}{8} \|\sqrt{\bar c^n}\eta_{jh}^{\bm q^n}\|_{\mathcal T_h}^2 +  C (\|\eta^{{u}^{n-1}}_{jh}\|_{\mathcal T_h}^2 + \Delta t^2 \| \partial_t^+\Pi_W^ju_j^n\|_{\mathcal T_h}^2),\\
	R_{34} &\le \frac{1-\alpha}{8} \| \sqrt{\bar c^n} \eta^{{\bm q}^{n-1}}_{jh}\|_{\mathcal T_h}^2 + \frac{\Delta t^2}4 \| \partial_t^+\bm{\Pi}_V^j\bm q_{j}^n\|_{\mathcal T_h}^2 \\
	&\quad  +C(\|\eta^{{u}^{n-1}}_{jh}\|_{\mathcal T_h}^2  + \Delta t^2 \| \partial_t^+\Pi_W^ju_j^n\|_{\mathcal T_h}^2),\\
	R_{35} &\le Ch^{2k+2}( |u_j^n|^2_{k+1}+|\bm q_j^n|^2_{k+1}) +C(\|\eta^{{u}^{n-1}}_{jh}\|_{\mathcal T_h}^2  + \Delta t^2 \| \partial_t^+\Pi_W^ju_j^n\|_{\mathcal T_h}^2).
	\end{align*}
	
	We add \eqref{eror_equa_1_NN} from $n=1$ to $n=N$, use the above inequalities  to get
	\begin{align}\label{error_1_NN}
	\begin{split}
	\hspace{0.2em}&\hspace{-0.2em}\max_{1\le n\le N}\|\eta^{{u}^n}_{jh}\|^2_{\mathcal{T}_h} +\Delta t\sum_{n=1}^N\|\sqrt{\bar c^n}\eta^{\bm{q}^{n}}_{jh}\|^2_{\mathcal{T}_h}\\
	&\le C\Delta t\sum_{n=1}^N \|\eta^{u^n}_{jh}\|^2_{\mathcal{T}_h} +C\sum_{n=1}^N
	(
	\Delta t^3\|\partial_t^+\Pi_W^ju_j^n\|^2_{\mathcal{T}_h}
	+\Delta t^3\|\partial_t^+\bm{\Pi}_V^j\bm q_j^n\|^2_{\mathcal{T}_h}
	)\\
	&\quad+C\sum_{n=1}^N
	(
	\Delta t\|\partial^+_t(u_j^n-\Pi_W^j u_j^n)\|^2_{\mathcal{T}_h}
	+\Delta t\|\partial^+_t u_j^n-\partial_t u_j^n\|^2_{\mathcal{T}_h})\\
	&\quad +  Ch^{2k+2} \sum_{n=1}^N \Delta t ( |u_j^n|^2_{k+1}+|\bm q_j^n|^2_{k+1}) + \|\eta^{u^0}_{jh}\|_{\mathcal T_h}^2+ \|\eta^{{\bm q}^0}_{jh}\|_{\mathcal T_h}^2.
	\end{split}
	\end{align}
	Now we move to bound the terms on the right side of the above inequality as follows,
	\begin{align*}
	\Delta t^3\sum_{n=1}^N\|\partial_t^+\Pi_W^ju_j^n\|^2_{\mathcal{T}_h}
	&=\Delta t\sum_{n=1}^N\int_{\Omega}\left[\int_{t^{n-1}}^{t^n}\partial_t\Pi_W^ju_j^n dt\right]^2 \\
	&\le C \Delta t^2 \|\partial_t\Pi_W^ju_j^n\|^2_{L^2(0,T;L^2(\Omega))},\\
	\Delta t^3\sum_{n=1}^N\|\partial_t^+\bm{\Pi}_V^j\bm q_j^n\|^2_{\mathcal{T}_h}
	&=\Delta t\sum_{n=1}^N\int_{\Omega}\left[\int_{t^{n-1}}^{t^n}\partial_t\bm{\Pi}_V^j\bm q_j^n dt\right]^2  \\
	&\le  C\Delta t^2 \|\partial_t\bm{\Pi}_V^j\bm q_j^n\|^2_{L^2(0,T;L^2(\Omega))},
	\end{align*}
	and
	\begin{align*}
	\Delta t\sum_{n=1}^N\|\partial^+_t(u_j^n-\Pi_W^j u_j^n)\|^2_{\mathcal{T}_h}
	&=\Delta t^{-1}\sum_{n=1}^N\int_{\Omega}\left[\int_{t^{n-1}}^{t^n}\partial_t(u_j^n-\Pi_W^j u_jdt)\right]^2 \\
	&\le C \|\partial_t(u_j^n-\Pi_W^j u_j)\|^2_{L^2(0,T;L^2(\Omega))},\\
	\Delta t\sum_{n=1}^N\|\partial^+_t u_j^n-\partial_t u_j^n\|^2_{\mathcal{T}_h}
	&= \Delta {t}^{-1}\sum_{n=1}^N\int_{\Omega}\left[\int_{t^{n-1}}^{t^n}(t-t^{n-1})\partial_{tt} u_jdt\right]^2\nonumber\\
	&\le C\Delta t^2 \|\partial_{tt} u_j\|^2_{L^2(0,T;L^2(\Omega))}.
	\end{align*}
	Gronwall's inequality and the estimates above applied to \eqref{error_1_NN}
	give the result.
\end{proof}

From \Cref{energy_quu}  and the estimate in \Cref{pro_error} we complete  the proof of \eqref{errror_qu} in \Cref{main_theorem}.

\subsection{Proof of \eqref{eror_ustar} in \Cref{main_theorem}}
To prove \eqref{eror_ustar} in \Cref{main_theorem}, we follow a similar strategy taken by Chen, Cockburn, Singler and Zhang \cite{ChenCockburnSinglerZhang1} and  introduce an HDG elliptic projection in \Cref{HDGellipticprojection}. We first bound the error between the solutions of the  HDG elliptic projection and the exact solution of the system \eqref{concection_pde}. Then we bound the error between the solutions of the HDG elliptic projection and  the Ensemble HDG problem \eqref{or}. A simple application of the triangle inequality then gives a bound on the error between the solutions of the Ensemble HDG problem and the system \eqref{concection_pde}. We note that the coefficients of the PDEs  are independent of time  throughout this section. Hence, we drop the superscript $n$ from  $c_j^n, \bm{\beta}_j^n$ and the ensemble means $\overline{c }^n,\overline{\bm \beta}^n$.
\subsubsection{HDG elliptic projection}
\label{HDGellipticprojection}
For any $t\in[0,T]$, let $(\overline{\bm q}_{jh},\overline{u}_{jh},\widehat{\overline u}_{jh})\in \bm V_h\times W_h\times M_h$ be the solutions of the following steady state problems
\begin{subequations}\label{projection}
	\begin{align}
	(c_j \overline{\bm{q}}_{jh},\bm{r}_h)_{\mathcal{T}_h}-(\overline u _{jh},\nabla\cdot \bm{r}_h)_{\mathcal{T}_h}+\langle\widehat{\overline u} _{jh},\bm{r}_h\cdot \bm n \rangle_{\partial{\mathcal{T}_h}} =
	-\left\langle
	g_j,\bm{r}_h\cdot \bm n \right\rangle_{{\mathcal{E}^{\partial}_h}}, \label{full_discretion_ensemble_steady1_a}\\
	(\nabla\cdot\overline{\bm{q}} _{jh}, v_h)_{\mathcal{T}_h}-\langle \overline{\bm{q}} _{jh}\cdot \bm{n},\widehat{v}_h\rangle_{\partial{\mathcal{T}_h}}
	+\langle \tau (\overline u _{jh}-{\widehat{\overline{u}}} _{jh}),v_h-\widehat{v}_h \rangle_{\partial\mathcal{T}_h}\nonumber\\
	+({\bm \beta}_j\cdot \nabla \overline{u}_{jh}, v_h)_{\mathcal{T}_h}-\left\langle{\bm \beta}_j\cdot\bm n, \overline{u}_{jh} \widehat{v}_h\right\rangle_{\partial\mathcal{T}_h}
	= (f _j-\Pi_W^j\partial_t u_j ,v_h)_{\mathcal{T}_h} + \left\langle \tau g_j,v_h\right\rangle_{\mathcal{E}^{\partial}_h},\label{full_discretion_ensemble_steady1_b}
	\end{align}
\end{subequations}
for all $(\bm r_h,v_h,\widehat{v}_h)\in \bm V_h\times W_h\times M_h$ and $j = 1,2,\cdots ,J$.

The proofs of the following estimates are given in \Cref{app}.
\begin{theorem}\label{error_dual}
	For any $t\in[0,T]$ and for all $j=1,2,\cdots, J$, we have
	\begin{subequations}
		\begin{align}
		\|\bm{\Pi}_{V}^j\bm q_j- \overline{\bm q}_{jh}\|_{\mathcal{T}_h} &\le C \mathcal A_j, \label{p1}\\
		\|\Pi_{W}^j {u}_j-\overline u_{jh}\|_{\mathcal{T}_h} &\le  Ch^{\min\{k,1\}}\mathcal A_j, \label{p2}\\
		\|\partial_t (\bm{\Pi}_{V}^j\bm q_j- \overline{\bm q}_{jh})\|_{\mathcal{T}_h} &\le C \mathcal B_j,\label{error_dual_bb}\\
		\|\partial_t(\Pi_W^j u_j - \overline{u}_{jh})\|_{\mathcal T_h}  &\le  Ch^{\min\{k,1\}} \mathcal B_j,\label{error_dual_c}\\
		\|\partial_{tt}(\Pi_W^j u_j - \overline{u}_{jh})\|_{\mathcal T_h}  &\le  Ch^{\min\{k,1\}} \mathcal C_j,\label{error_dual_d}
		\end{align}
	\end{subequations}
	where
	\begin{align*}
	\mathcal A_j &= \|u_j - \Pi_W^j u_j\|_{\mathcal T_h}+\|\bm q_j - \bm\Pi_V^j \bm q_j\|_{\mathcal T_h} +\|\partial_tu_j - \Pi_W^j \partial_tu_j\|_{\mathcal T_h},\\
	\mathcal B_j & = \|\partial_{t}u_j - \Pi_W^j \partial_{t}u_j\|_{\mathcal T_h}+
	\|\partial_t\bm q_j - \bm\Pi_V^j\partial_t\bm q_j\|_{\mathcal T_h} +\|\partial_{tt}u_j - \Pi_W^j \partial_{tt}u_j\|_{\mathcal T_h},\\
	\mathcal C_j &=\|\partial_{tt}u_j - \Pi_W^j \partial_{tt}u_j\|_{\mathcal T_h}+
	\| \partial_{tt}\bm q_j - \bm\Pi_V^j \partial_{tt} \bm q_j\|_{\mathcal T_h} +\|\partial_{ttt}u_j - \Pi_W^j \partial_{ttt}u_j\|_{\mathcal T_h}.
	\end{align*}
\end{theorem}

Note that \Cref{error_dual} bounds the error between  the  HDG elliptic projection of the solutions  and the exact solutions of the system \eqref{concection_pde}. In the next three steps, we are going to bound the error between the HDG elliptic projection of the ensemble solutions and the solutions of  the Ensemble HDG problem \eqref{or}.

\subsubsection{The equations of the projection of the errors}
\begin{lemma} \label{error_lemma}
	For $e^{u^n}_{jh}=u_{jh}^n-{\overline{u}}_{jh}^n$, $e^{\bm q^n}_{jh}=\bm q_{jh}^n-{\overline{\bm q}}_{jh}^n$  and $ e^{\widehat{u}^n}_{jh}=\widehat{u}_{jh}^n-{{\widehat{\overline{u}}}}_{jh}^n$, for all $j = 1,2,\cdots,J$, we have  the following  error equations
	\begin{subequations}\label{error1}
		\begin{align}
		(\bar ce_{jh}^{\bm q^n},\bm{r}_h)_{\mathcal{T}_h}
		-(e^{u^n}_{jh},\nabla\cdot \bm{r}_h)_{\mathcal{T}_h}+\langle e^{\widehat{u}^n}_{jh},\bm{r}_h\cdot \bm n \rangle_{\partial{\mathcal{T}_h}} =
		( ( \bar c-c_j)(\bm{q}^{n-1}_{jh}-
		\overline{\bm{q}}^{n}_{jh}),\bm{r}_h)_{\mathcal{T}_h},
		\label{error_a}
		\end{align}
		and 
		\begin{align}
		&\quad(\partial^+_te^{u^n}_{jh},v_h)_{\mathcal T_h}+(\nabla\cdot e_{jh}^{\bm q^n}, v_h)_{\mathcal{T}_h}-\langle e^{\bm{q}^n}_{jh}\cdot \bm{n},\widehat{v}_h\rangle_{\partial{\mathcal{T}_h}}
		+( \overline{\bm \beta}\cdot \nabla e^{u^n}_{jh}, v_h)_{\mathcal{T}_h} \nonumber\\
		&\quad-\langle \overline{\bm \beta}\cdot\bm n, e^{{u}^n}_{jh} \widehat{v}_h\rangle_{\partial\mathcal{T}_h}+
		\langle \tau ( e^{u^n}_{jh}-e^{\widehat{u}}_{jh}),v_h-\widehat{v}_h \rangle_{\partial\mathcal{T}_h} -	(\partial^+_t\overline{u}_{jh}^n-\partial_t\Pi_W^j u_j^n, v_h)_{\mathcal{T}_h}
		\nonumber\\
		&=( (\overline{\bm \beta}-\bm{\beta}_j)\cdot \nabla (u_{jh}^{n-1}
		-\overline{u}_{jh}^{n}), v_h)_{\mathcal{T}_h} -\langle (\overline{\bm \beta}_j-\bm{\beta}_j)\cdot\bm n, ({u}_{jh}^{n-1}-\overline{u}^n_{jh}) \widehat{v}_h\rangle_{\partial\mathcal{T}_h},\label{full_discretion_ensemble_b}
		\end{align}
	\end{subequations}
	for all $(\bm r_h,v_h,\widehat{v}_h)\in \bm V_h\times W_h\times M_h$ and $n=1,2,\cdots,N$.
\end{lemma}

The proof of \Cref{error_lemma} follows immediately by simply subtracting \Cref{projection} from \Cref{or}.

\subsubsection{Energy argument}
\begin{lemma}\label{energy_qu}
	If condition \eqref{condition_c} and the elliptic regularity inequality \eqref{Dual_PDE1} holds, then we have the following error estimate:
	\begin{equation}
	\begin{split}
	\max_{1\le n\le N}\|e^{{u}^n}_{jh}\|_{\mathcal{T}_h}
	\le C\left(h^{k+1+\min\{k,1\}}+\Delta t\right).
	\end{split}
	\end{equation}
\end{lemma}
\begin{proof}
	The following proof  is  similar to  the proof  in \Cref{Proofof16};  to make the proof self-contained, we include the details here. We take $(\bm r_h,v_h,\widehat{v}_h)=(e^{\bm{q}^{n}}_{jh},e^{u^n}_{jh},e^{{\widehat{u}}^n}_{jh} )$
	in \eqref{error1},  use the identity \eqref{polarization} and add \Cref{error_a} - \Cref{full_discretion_ensemble_b} together to get
	\begin{align}\label{29}
	\begin{split}
	\hspace{1em}&\hspace{-1em} \frac{\|e^{{u}^n}_{jh}\|^2_{\mathcal{T}_h}-\|e^{{u}^{n-1}}_{jh}\|^2_{\mathcal{T}_h}}{2\Delta t}
	+\frac{\|e^{{u}^n}_{jh}-e^{{u}^{n-1}}_{jh}\|^2_{\mathcal{T}_h}}{2\Delta t} +\|\sqrt{\bar c}e^{\bm{q}^{n}}_{jh}\|^2_{\mathcal{T}_h}
	+\|\sqrt{\tau}(e^{u^n}_{jh}-e^{{\widehat{u}}^n}_{jh})\|^2_{\partial \mathcal{T}_h}\\
	&= - ( \overline{\bm \beta}\cdot \nabla e^{u^n}_{jh}, e^{u^n}_{jh})_{\mathcal{T}_h} + \langle \overline{\bm \beta}\cdot\bm n, e^{{u}^n}_{jh} e^{\widehat{u}^n}_{jh}\rangle_{\partial\mathcal{T}_h}+
	((\overline c-c_j)(\bm q_{jh}^{n-1}-\overline{\bm q}_{jh}^n),e^{\bm{q}^{n}}_{jh})_{\mathcal{T}_h}\\
	&\quad+
	(\partial^+_t\overline{u}_{jh}^n-\partial_t\Pi_W^j u_j^n,e^{u^n}_{jh})_{\mathcal{T}_h}+( (\overline{\bm \beta}-\bm{\beta}_j)\cdot \nabla (u_{jh}^{n-1}
	-\overline{u}_{jh}^{n}), e^{u^n}_{jh})_{\mathcal{T}_h}\\
	&\quad-\langle (\overline{\bm \beta}-\bm{\beta}_j)\cdot\bm n, ({u}_{jh}^{n-1}-\overline{u}^n_{jh}) e^{\widehat{u}^n}_{jh}\rangle_{\partial\mathcal{T}_h}.
	\end{split}
	\end{align}
	By Green's formula and the fact
	$\langle (\overline{\bm \beta}\cdot\bm n) e^{\widehat{u}^n}_{jh},  e^{\widehat{u}^n}_{jh}\rangle_{\partial\mathcal{T}_h}=0$,
	we have
	\begin{align*}
	( \overline{\bm \beta}\cdot \nabla e^{u^n}_{jh}, e^{u^n}_{jh})_{\mathcal{T}_h}
	-\langle \overline{\bm \beta}\cdot\bm n, e^{{u}^n}_{jh} e^{\widehat{u}^n}_{jh}\rangle_{\partial\mathcal{T}_h}
	\le \frac{1}{2}\|\sqrt{|\overline{\bm \beta}\cdot\bm n|}
	(e^{u^n}_{jh}-e^{\widehat{u}^n}_{jh})\|^2_{\partial\mathcal{T}_h}.
	\end{align*}
	Condition \eqref{tau} and equality  \eqref{29} give
	\begin{align*}
	\hspace{0.2em}&\hspace{-0.2em}\frac{\|e^{{u}^n}_{jh}\|^2_{\mathcal{T}_h}-\|e^{{u}^{n-1}}_{jh}\|^2_{\mathcal{T}_h}}{2\Delta t}
	+\frac{\|e^{{u}^n}_{jh}-e^{{u}^{n-1}}_{jh}\|^2_{\mathcal{T}_h}}{2\Delta t} +\|\sqrt{\bar c}e^{\bm{q}^{n}}_{jh}\|^2_{ \mathcal{T}_h}
	+\frac 1 2\|\sqrt{\tau}(e^{u^n}_{jh}-e^{{\widehat{u}}^n}_{jh})\|^2_{\partial\mathcal{T}_h}\\
	& \le 
	((\overline c-c_j)(\bm q_{jh}^{n-1}-\overline{\bm q}_{jh}^n),e^{\bm{q}^{n}}_{jh})_{\mathcal{T}_h}+
	(\partial^+_t\overline{u}_{jh}^n-\partial_t\Pi_W^j u_j^n,e^{u^n}_{jh})_{\mathcal{T}_h}\\
	&~~+( (\overline{\bm \beta}-\bm{\beta}_j)\cdot \nabla (u_{jh}^{n-1}
	-\overline{u}_{jh}^{n}), e^{u^n}_{jh})_{\mathcal{T}_h}-\langle (\overline{\bm \beta}-\bm{\beta}_j)\cdot\bm n, ({u}_{jh}^{n-1}-\overline{u}^n_{jh}) e^{\widehat{u}^n}_{jh}\rangle_{\partial\mathcal{T}_h}\\
	&= T_1 + T_2 + T_3.
	\end{align*}
	Next,  we estimate $\{T_i\}_{i=1}^3$. By the condition \eqref{condition_c}, there exist $0<\alpha <1$ such that 
	\begin{align*}
	T_1&=((\overline c-c_j)(e_{jh}^{\bm q^{n-1}}
	-\Delta t\partial_t^+\overline{\bm q}_{jh}^n),e^{\bm{q}^{n}}_{jh})_{\mathcal{T}_h}
	\nonumber\\
	&\le \frac{\alpha}{2}\left(\|\sqrt{\bar c}e^{\bm{q}^{n}}_{jh}\|^2_{\mathcal{T}_h}
	+\|\sqrt{\bar c}e^{\bm{q}^{n-1}}_{jh}\|^2_{\mathcal{T}_h}
	\right)+C\Delta t^2\|\partial_t^+\overline{\bm q}_{jh}^n\|^2_{\mathcal{T}_h},\nonumber\\
	T_2&=(\partial^+_t(\overline{u}_{jh}^n-\Pi_W^j u^n_j)
	+\partial^+_t\Pi_W^ju^n_j-\partial_t\Pi_W^j u_j^n,e^{u^n}_{jh})_{\mathcal{T}_h}\nonumber\\
	&\le C\left( \|\partial^+_t(\overline{u}_{jh}^n-\Pi_W^ju^n_j)\|^2_{\mathcal{T}_h}
	+\|\partial^+_t\Pi_W^j u^n_j-\partial_t\Pi_W^j u_j^n\|^2_{\mathcal{T}_h}
	+\|e^{u^n}_{jh}\|^2_{\mathcal{T}_h} \right).
	\end{align*}
	To treat the term $T_3$, we use the technique in the proof of  \Cref{energy_quu},  where we treat the term $R_3$.  For  $\bm \gamma = \overline{\bm \beta}-\bm{\beta}_j$, we have 
	\begin{align*}
	T_3  &\le \sum_{K\in \mathcal T_h} \|\bm \gamma - \bm\Pi_0 \bm \gamma\|_{\infty, K} \|\nabla (u_{jh}^{n-1}
	-\overline{u}_{jh}^{n})\|_K \| e^{u^n}_{jh}\|_K\\
	&\quad  +  \sum_{K\in \mathcal T_h} \|\bm \gamma - \bm\Pi_0 \bm \gamma\|_{\infty, \partial K} \|u_{jh}^{n-1}
	-\overline{u}_{jh}^{n}\|_{\partial K} (\|  e^{\widehat{u}^n}_{jh}  - e^{{u}^n}_{jh}\|_{\partial K} + \|e^{{u}^n}_{jh}\|_{\partial K})\\
	& \quad + \|\bm \Pi_0 \bm \gamma\|_{\infty,\mathcal T_h} \|\bar c e_{jh}^{\bm q^n}\|_{\mathcal T_h} \|u_{jh}^{n-1}
	-\overline{u}_{jh}^{n}\|_{\mathcal T_h}\\
	&\quad + \| (\bar c-c_j)\bm \Pi_0\bm \gamma\|_{\infty,\mathcal T_h} \| \bm{q}^{n-1}_{jh}-
	\overline{\bm{q}}^{n}_{jh}\|_{\mathcal T_h}  \|u_{jh}^{n-1}
	-\overline{u}_{jh}^{n}\|_{\mathcal T_h}\\
	& = T_{31} + T_{32} + T_{33} + T_{34}.
	\end{align*}
	For $T_{31}$, use the local inverse inequality:
	\begin{align*}
	T_{31} \le C (\|e^{{u}^{n-1}}_{jh}\|_{\mathcal T_h}^2 + \Delta t^2 \| \partial_t^+\overline{u}_{jh}^{n}\|_{\mathcal T_h}^2 + \|e^{{u}^{n}}_{jh}\|_{\mathcal T_h}^2). 
	\end{align*}
	Apply the trace inequality and inverse inequality for the term $T_{32}$ to give
	\begin{align*}
	T_{32}  &\le   C  \sum_{K\in \mathcal T_h} h_K \|\bm \gamma \|_{1,\infty, K} h_K^{-1/2} \|u_{jh}^{n-1}
	-\overline{u}_{jh}^{n}\|_K (\|  e^{\widehat{u}^n}_{jh}  - e^{{u}^n}_{jh}\|_{\partial K} +h_K^{-1/2} \|e^{{u}^n}_{jh}\|_{K})\\
	&\le C (\|e^{{u}^{n-1}}_{jh}\|_{\mathcal T_h}^2 + \Delta t^2 \| \partial_t^+\overline{u}_{jh}^{n}\|_{\mathcal T_h}^2 + \|e^{{u}^{n}}_{jh}\|_{\mathcal T_h}^2) + \frac 1 4  \| \sqrt{\tau}(e^{\widehat{u}^n}_{jh}  - e^{{u}^n}_{jh})\|_{\partial \mathcal T_h}^2. 
	\end{align*}
	For the terms $T_{33}$ and $T_{34}$, use  Young's inequality to obtain
	\begin{align*}
	T_{33}&\le \frac{1-\alpha}{8} \|\sqrt{\bar c} e_{jh}^{\bm q^n}\|_{\mathcal T_h}^2 +  C (\|e^{{u}^{n-1}}_{jh}\|_{\mathcal T_h}^2 + \Delta t^2 \| \partial_t^+\overline{u}_{jh}^{n}\|_{\mathcal T_h}^2),\\
	T_{34} &\le \frac{1-\alpha}{8} \| \sqrt{\bar c} e^{{\bm q}^{n-1}}_{jh}\|_{\mathcal T_h}^2 + \frac{\Delta t^2}4 \| \partial_t^+\overline{\bm q}_{jh}^{n}\|_{\mathcal T_h}^2  +C(\|e^{{u}^{n-1}}_{jh}\|_{\mathcal T_h}^2  + \Delta t^2 \| \partial_t^+\overline{u}_{jh}^{n}\|_{\mathcal T_h}^2). 
	\end{align*}
	
	We add \eqref{29} from $n=1$ to $n=N$,
	and use the above inequalities  to get
	\begin{align}\label{error_1_N}
	\begin{split}
	\hspace{0.2em}&\hspace{-0.2em}\max_{1\le n\le N}\|e^{{u}^n}_{jh}\|^2_{\mathcal{T}_h}
	+\sum_{n=1}^N\|e^{{u}^n}_{jh}-e^{{u}^{n-1}}_{jh}\|^2_{\mathcal{T}_h}+\Delta t\sum_{n=1}^N\|\sqrt{\bar c}e^{\bm{q}^{n}}_{jh}\|^2_{\mathcal{T}_h}+\|\sqrt{\tau}(e^{u^n}_{jh}-e^{{\widehat{u}}^n}_{jh})\|^2_{\mathcal{T}_h}\\
	&\le C\Delta t\sum_{n=1}^N \|e^{u^n}_{jh}\|^2_{\mathcal{T}_h} +C\sum_{n=1}^N
	(
	\Delta t^3\|\partial_t^+\overline{u}_{jh}^n\|^2_{\mathcal{T}_h}
	+\Delta t^3\|\partial_t^+\overline{\bm q}_{jh}^n\|^2_{\mathcal{T}_h}
	)\\
	&\quad+C\sum_{n=1}^N
	(
	\Delta t\|\partial^+_t(\overline{u}_{jh}^n-\Pi_W^j u_j^n)\|^2_{\mathcal{T}_h}
	+\Delta t\|\partial^+_t\Pi_W^j u_j^n-\partial_t\Pi_W^j u_j^n\|^2_{\mathcal{T}_h}
	) \\
	&\quad +C\|e^{\bm{q}^{0}}_{jh}\|^2_{\mathcal{T}_h} +C \|e^{{u}^0}_{jh}\|^2_{\mathcal{T}_h}.
	\end{split}
	\end{align}
	Now we move to bound the terms on the right side of the above inequality as follows,
	\begin{align*}
	&\Delta t^3\sum_{n=1}^N\|\partial_t^+\overline{u}_{jh}^n\|^2_{\mathcal{T}_h}
	=\Delta t\sum_{n=1}^N\int_{\Omega}\left[\int_{t^{n-1}}^{t^n}\partial_t\overline{u}_{jh}^n dt\right]^2  \le C\Delta t^2 \|\partial_t\overline{u}_{jh}^n\|^2_{L^2(0,T;L^2(\Omega))},\\
	&\Delta t^3\sum_{n=1}^N\|\partial_t^+\overline{\bm q}_{jh}^n\|^2_{\mathcal{T}_h}
	=\Delta t\sum_{n=1}^N\int_{\Omega}\left[\int_{t^{n-1}}^{t^n}\partial_t\overline{\bm q}_{jh}^n dt\right]^2  \le C\Delta t^2 \|\partial_t\overline{\bm q}_{jh}^n\|^2_{L^2(0,T;L^2(\Omega))},\\
	&\Delta t\sum_{n=1}^N\|\partial^+_t(\overline{u}_{jh}^n-\Pi_W^j u_j^n)\|^2_{\mathcal{T}_h}\le C \|\partial_t(\overline{u}_{jh}-\Pi_W^j u_j)\|^2_{L^2(0,T;L^2(\Omega))},\\
	&\Delta t\sum_{n=1}^N\|\partial^+_t\Pi_W^j u_j^n-\partial_t\Pi_W^j u_j^n\|^2_{\mathcal{T}_h}
	\le C\Delta t^2 \|\partial_{tt}\Pi_W^j u_j\|^2_{L^2(0,T;L^2(\Omega))}.
	\end{align*}
	Gronwall's inequality and the estimates above applied to \eqref{error_1_N}
	give the result.
\end{proof}

\subsubsection{Superconvergence error estimates by postprocessing}
The following element-by-element postprocessing is defined in \cite{Cockburn_Gopalakrishnan_Sayas_Porjection_MathComp_2010}: Find $u_{jh}^{n\star} \in \mathcal P^{k+1}(K)$ such that for all $(z_h, w_h)\in \mathcal [\mathcal P^{k+1}(K)]^{\perp}\times\mathcal{P}^{0}(K) $ 
\begin{subequations}\label{post_process_1}
	\begin{align}
	(\nabla u_{jh}^{n\star},\nabla z_h)_K&=-({c}_j\bm q^n_{jh},\nabla z_h)_K,\label{post_process_1_a}\\
	(u_{jh}^{n\star},w_h)_K&=(u_h,w_h)_K,\label{post_process_1_b}
	\end{align}
\end{subequations}
where $\mathcal [\mathcal P^{k+1}(K)]^{\perp} = \{z_h\in \mathcal P^{k+1}(K)| (z_h,1)_K = 0\} $. 
\begin{lemma}\label{super_con}
	For any $t\in[0,T]$ and   $k\ge 1$,  we have the following error estimate for the postprocessed solution: 
	\begin{align*}
	\|\Pi_{k+1} u_j^n - u_{jh}^{n\star}\|_{\mathcal T_h}  &\le C  \|\Pi_W^j u^n_j- u^n_{jh}\|_{\mathcal{T}_h}+Ch\|c_j(\bm{q}^n_{jh}-\bm{q}^n_j) \|_{\mathcal{T}_h} \\
	&\quad +Ch\|\nabla(u_j^n - \Pi_{k+1} u_j^n)\|_{\mathcal{T}_h}.
	\end{align*}
\end{lemma}

\begin{proof}
	By the properties of $\Pi_W$ and $\Pi_{k+1}$, we obtain
	\begin{align*}
	(\Pi_W^j u_j^n,w_0)_K &=(u_j^n,w_0)_K,\quad \text{ for all } w_0\in \mathcal{P}^{0}(K),\\
	(\Pi_{k+1} u_j^n,w_0)_K&=(u_j^n,w_0)_K, \quad \text{ for all } w_0\in \mathcal{P}^{0}(K).
	\end{align*}
	Hence, for all $w_{0}\in \mathcal P^{0}(K)$, we have
	\begin{align*}
	(\Pi_W u_j^n-\Pi_{k+1} u_j^n, w_{0})_K = 0.
	\end{align*}
	
	Let $e_{jh}^n=u_{jh}^{n\star} - u_{jh}^n+\Pi_W^j u_j^n-\Pi_{k+1} u_j^n$.  \Cref{post_process_1} and an inverse inequality give 
	\begin{align*}
	\|\nabla e^n_{jh}\|_K^2&=(\nabla (u_{jh}^{n\star} - u_{jh}),\nabla e^n_{j,h} )_K+( \nabla (\Pi_W^ju^n_j -\Pi_{k+1} u^n_j),\nabla e^n_{jh} )_K \nonumber\\
	&=(-\nabla u^n_{jh}-{c}_j \bm{q}^n_{jh},\nabla e^n_{jh} )_K+(  \nabla (\Pi_W^j u^n_j-\Pi_{k+1} u^n_j),\nabla e^n_{jh} )_K \nonumber\\
	&=(\nabla (\Pi_W^ju^n_j - u^n_{jh})-(\bm{q}^n_{jh}-\bm{q}^n_j)  +\nabla (u^n_j-\Pi_{k+1} u^n_j),\nabla e^n_{jh})_K.
	\end{align*}
	This implies
	\begin{align}\label{H1}
	\|\nabla e^n_{jh}\|_K \le C  (h_K^{-1}\|\Pi_W^j u^n_j- u^n_{jh}\|_K+\|c_j (\bm{q}^n_{jh}-\bm{q}^n_j) \|_K +\|\nabla(u_j^n - \Pi_{k+1} u_j^n)\|_{K}).
	\end{align}
	Since $(e_h,1)_K=0$,  apply the Poincar\'{e} inequality and the estimate \eqref{H1} to give
	\begin{align*}
	\|e_{jh}^n\|_K&\le C h_K \|\nabla e^n_{jh}\|_K \\
	&\le C (\|\Pi_W^j u^n_j- u^n_{jh}\|_K+h_K\|c_j(\bm{q}^n_{jh}-\bm{q}^n_j) \|_K +h_K\|\nabla(u_j^n - \Pi_{k+1} u_j^n)\|_{K} ).
	\end{align*}
	Hence, we have
	\begin{align*}
	\|\Pi_{k+1} u^n_{j} - u_{jh}^{n\star}\|_{\mathcal T_h}&\le\|\Pi_{k+1} u^n_j -\Pi_W^j u^n_j- u_{jh}^{n\star} + u_{jh}^n\|_{\mathcal T_h} + \|\Pi_W^j u^n_j-u^n_{jh}\|_{\mathcal T_h} \nonumber\\
	&  \le  C\|\Pi_W^j u^n_j- u^n_{jh}\|_{\mathcal{T}_h}+Ch\|c_j(\bm{q}^n_{jh}-\bm{q}^n_j) \|_{\mathcal{T}_h}\\
	&\quad +Ch\|\nabla(u_j^n - \Pi_{k+1} u_j^n)\|_{\mathcal{T}_h}.
	\end{align*}
\end{proof}

From \Cref{super_con}  and the estimate in \eqref{lemmainter_orthoo} we complete the proof of \eqref{eror_ustar} in \Cref{main_theorem}.

\section{Numerical experiments}
\label{Numericalexperiments}
In this section, we present some numerical tests of the Ensemble HDG method  for parameterized convection diffusion PDEs.  Although we derived the a priori error estimates for diffusion dominated  problems, we also present numerical results for the convection dominated case to show the performance of the Ensemble HDG method for the convection  dominated diffusion problems. For all examples, we take 
$\tau=1+\max_{1\le j\le J}\|\bm{\beta}_j\|_{0,\infty}$ so that \eqref{tau} is satisfied, the coefficients $c_j$ satisfy the condition \eqref{condition_c},  and a group of simulations are  considered containing  $J = 3$ members. Let $Eu_j$ be the error bewteen the exact solution $u_j$ at the final time $T=1$ and the Ensemble HDG solution $u_{jh}^N$, i.e., $Eu_j = \|u_j^N -u_{jh}^N\|_{\mathcal T_h}$. Let 
\begin{align*}
E\bm q _j = \sqrt{\Delta t\sum_{n=1}^N \|\bm q_j^n - \bm q_{jh}^n\|^2_{\mathcal T_h}}, \quad \textup{and}\quad E u_j^\star = \sqrt{\Delta t\sum_{n=1}^N \| u_j^n -  u_{jh}^{n\star}\|^2_{\mathcal T_h}}.
\end{align*}

\begin{example}\label{example1}
	We first test the convergence rate of the Ensemble HDG method for  diffusion dominated PDEs on a square domain $\Omega = [0,1]\times[0,1]$. The data is chosen as
	\begin{gather*}
	c_1 = 0.26959, \ c_2 = 0.26633, \ c_3 = 0.30525,\\
	\bm \beta_1 = 1.6797[y,x], \ \bm \beta_2 = 1.6551[y,x], \ \bm \beta_3 = 1.1626[y,x], \\
	u_j = \sin(t)\sin(x)\sin(y)/j,  \ \bm q_j = -1/c_j \nabla u_j, \  j=1,2,3,
	\end{gather*}
	and the initial conditions, boundary conditions, and source terms are chosen to match the exact solution of  \Cref{concection_pde}.
	
	In order to confirm our theoretical results, we take $\Delta t = h$ when $k=0$ and $\Delta t = h^3$ when $k=1$. The approximation errors of the Ensemble HDG method  are listed in \Cref{table_1} and the observed convergence rates match our theory. 	
	\begin{table}[H]
		\caption{History of convergence for Example \ref{example1}.}\label{table_1}
			\begin{tabular}{c|c|c|c|c|c|c|c}
				\Xhline{0.1pt}

				\multirow{2}{*}{Degree}
				&\multirow{2}{*}{$\frac{h}{\sqrt{2}}$}	
				&\multicolumn{2}{c|}{$E {\bm q}_1$}	
				&\multicolumn{2}{c|}{$E u_1$}	
				&\multicolumn{2}{c}{$Eu_1^\star$}	\\
				\cline{3-8}
				& &Error &Rate
				&Error &Rate
				&Error &Rate
				\\
				\cline{1-8}
				\multirow{5}{*}{ $k=0$}
				&$2^{-1}$	&8.5356E-01	    &	       &8.0704E-02	    &	         &1.3681E-01	&\\
				&$2^{-2}$	&5.3683E-01	    &0.67	&4.6752E-02      &0.79 	   &5.7997E-02 	&1.24\\
				&$2^{-3}$	&2.9377E-01 	&0.87 	& 2.4599E-02 	 &0.93 	   &2.6288E-02	&1.14\\
				&$2^{-4}$	&1.5300E-01 	&0.94 	&  1.2677E-02	 &0.96 	   &1.2902E-02	&1.03\\
				&$2^{-5}$	&7.8021E-02     &0.97	&  6.4474E-03    &0.98 	   &6.4760E-03 	&0.99\\

				\cline{1-8}
				\multirow{5}{*}{$k=1$}
				&$2^{-1}$	&2.6429E-01 & 	       & 4.2641E-02	& 	      & 4.3413E-02     &\\
				&$2^{-2}$	&7.5086E-02&1.82	&1.0472E-02&2.03	&6.1017E-03 	&2.83\\
				&$2^{-3}$	& 1.9707E-02 	&1.93	&2.6345E-03	&1.99	&7.9146E-04  &2.95\\
				&$2^{-4}$	& 5.0211E-03	&1.97	& 6.6870E-04 	&1.98	&1.0026E-04	&2.98\\
				&$2^{-5}$	& 1.2653E-03&1.99	&1.6896E-04	&1.98	& 1.2598E-05 	&2.99\\

				\Xhline {0.1 pt}

			\end{tabular}

			\begin{tabular}{c|c|c|c|c|c|c|c}
				\Xhline{1pt}

				\multirow{2}{*}{Degree}
				&\multirow{2}{*}{$\frac{h}{\sqrt{2}}$}	
				&\multicolumn{2}{c|}{$E {\bm q}_2$}	
				&\multicolumn{2}{c|}{$E u_2$}	
				&\multicolumn{2}{c}{$Eu_2^\star$}	\\
				\cline{3-8}
				& &Error &Rate
				&Error &Rate
				&Error &Rate
				\\
				\cline{1-8}
				\multirow{5}{*}{ $k=0$}
				&$2^{-1}$	&8.5466E-01 	&	       &8.1522E-02	   &	      &1.3739E-01	 &\\
				&$2^{-2}$	&5.3907E-01	    &0.66	&4.8107E-02     &0.76 	&5.9168E-02	   &1.22\\
				&$2^{-3}$	&2.9567E-01	    &0.87 	&2.5614E-02     &0.91 	&2.7258E-02    &1.12\\
				&$2^{-4}$	&1.5420E-01	    &0.94 	&1.3277E-02     &0.95 	&1.3495E-02    &1.01\\
				&$2^{-5}$	&7.8696E-02 	&0.97 	&6.7714E-03 	&0.97 	&6.7992e-03	   &0.99\\

				\cline{1-8}
				\multirow{5}{*}{$k=1$}
				
				&$2^{-1}$	&2.6577E-01 	&	    & 4.2796E-02	&	    & 4.3973E-02	&\\
				&$2^{-2}$	& 7.5666E-02&1.81 	& 1.0405E-02&2.04	&  6.2024E-03 	&2.83\\
				&$2^{-3}$	&1.9879E-02	&1.93	&2.6069E-03	&2.00 	&8.0552E-04	&2.94\\
				&$2^{-4}$	& 5.0673E-03	&1.97	& 6.6105E-04	&1.98	&1.0209E-04	&2.98\\
				&$2^{-5}$	&1.2772E-03 	&1.99	&1.6699E-04 	&1.99	&1.2832E-05	&2.99\\

				\Xhline{0.1 pt}

			\end{tabular}

			\begin{tabular}{c|c|c|c|c|c|c|c}
				\Xhline{1pt}

				\multirow{2}{*}{Degree}
				&\multirow{2}{*}{$\frac{h}{\sqrt{2}}$}	
				&\multicolumn{2}{c|}{$E {\bm q}_3$}	
				&\multicolumn{2}{c|}{$E u_3$}	
				&\multicolumn{2}{c}{$Eu_3^\star$}	\\
				\cline{3-8}
				& &Error &Rate
				&Error &Rate
				&Error &Rate
				\\
				\cline{1-8}
				\multirow{5}{*}{ $k=0$}
				&$2^{-1}$	&8.0839E-01 &	       &3.4525E-02	   &	      &1.1145E-01     &\\
				&$2^{-2}$	&5.0993E-01	&0.66	&2.2025E-02	    &0.65 	&3.9756E-02     &1.49\\
				&$2^{-3}$	&2.7915E-01	&0.87 	&1.2282E-02     &0.84 	& 1.5196E-02	&1.39\\
				&$2^{-4}$	&1.4529E-01 &0.94 	&6.5117E-03     &0.92 	&6.9102E-03     &1.14\\
				&$2^{-5}$	&7.4042E-02	&0.97 	& 3.3567E-03    &0.96 	& 3.4076E-03	&1.02\\

				\cline{1-8}
				\multirow{5}{*}{$k=1$}
				
				&$2^{-1}$	&2.4988E-01	&	    & 4.0593E-02&	    &3.9155E-02	&\\
				&$2^{-2}$	&6.9685E-02&1.84	&1.1221E-02	&1.86	&5.5006E-03	&2.83\\
				&$2^{-3}$	&1.8087E-02&1.95	&2.9375E-03&1.93 	&7.1138E-04	&2.95\\
				&$2^{-4}$	& 4.5831E-03&1.98 	&7.5247E-04&1.96 	&8.9813E-05	&2.99\\
				&$2^{-5}$	& 1.1520E-03&1.99 	&1.9046E-04	&1.98 	&1.1261E-05	&3.00\\

				\Xhline{0.1 pt}

			\end{tabular}

	\end{table}

\end{example}

\begin{example}\label{example2}
	Next, we perform  Ensemble HDG computations for the convection  dominated case with exact solutions haveing interior layers. But we do not attempt to compute convergence rates here; instead for illustration we plot all the ensemble members $\{u_{jh}\}_{j=1}^3$ at the final time $T= 0.1$ and also plot the exact solution for comparsion.  We can see that  the Ensemble HDG method is able to capture the very sharp interior layers in the solution with almost no oscillatory behavior, see e.g. \Cref{fig:u1,fig:u2,fig:u3}.
	
	The domain is $\Omega = [0,1]\times[0,1]$ and it is uniformly partition into $131072$ triangles ($h = \sqrt{2}/256$) and also  $\Delta t = 1/1000$. The  initial conditions, boundary conditions, and source terms are chosen to match  \Cref{concection_pde} for  the data 
	\begin{gather*}
	c_1 =10^{4}, \ c_2 =2 \times 10^{4}, \ c_3 =3 \times 10^{4},\\
	\bm \beta_1 = [2,3], \ \bm \beta_2 = [3,4], \ \bm \beta_3 =[4,5], 
	\end{gather*}
	and the exact solutions $\{u_j\}_{j=1}^3$ are chosen as
	\begin{align*}
	u_1 &= \sin(t)x(1-x)y(1-y)\left[\frac 1 2 +  \frac{{\arctan 2 \sqrt{c_1} \left(\frac{1}{12} - \left(x-\frac{1}{3}\right)^2 - \left(y-\frac 1 {2}\right)^2 \right)}}{\pi} \right],\\
	u_2 &= \sin(t)x(1-x)y(1-y)\left[\frac 1 2 +  \frac{{\arctan 2 \sqrt{c_2} \left(\frac{1}{14} - \left(x-\frac{1}{2}\right)^2 - \left(y-\frac 1 {3}\right)^2 \right)}}{\pi} \right],\\
	u_3 &= \sin(t)x(1-x)y(1-y)\left[\frac 1 2 +  \frac{{\arctan 2 \sqrt{c_3} \left(\frac{1}{16} - \left(x-\frac{1}{2}\right)^2 - \left(y-\frac 1 {2}\right)^2 \right)}}{\pi} \right].
	\end{align*}
	
	\begin{figure}
		\centerline{
			\hbox{\includegraphics[width=2.5in]{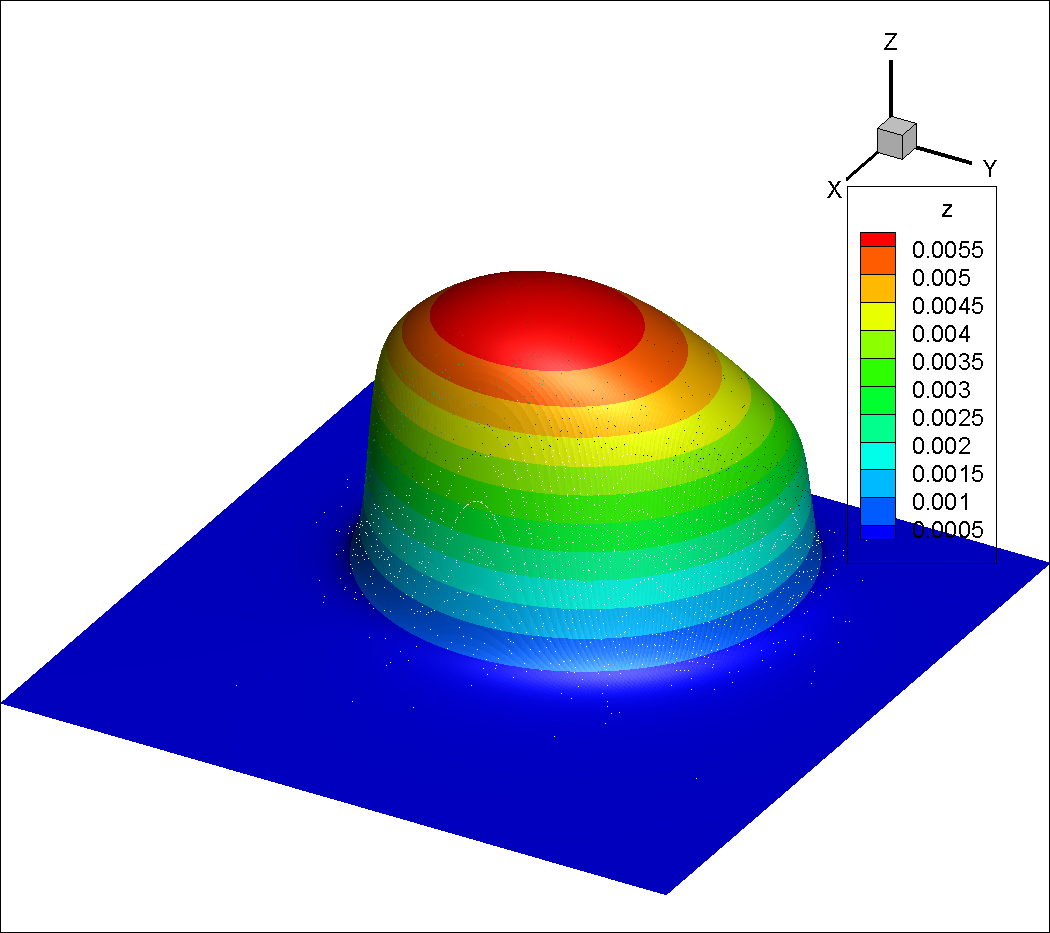}}
			\hbox{\includegraphics[width=2.5in]{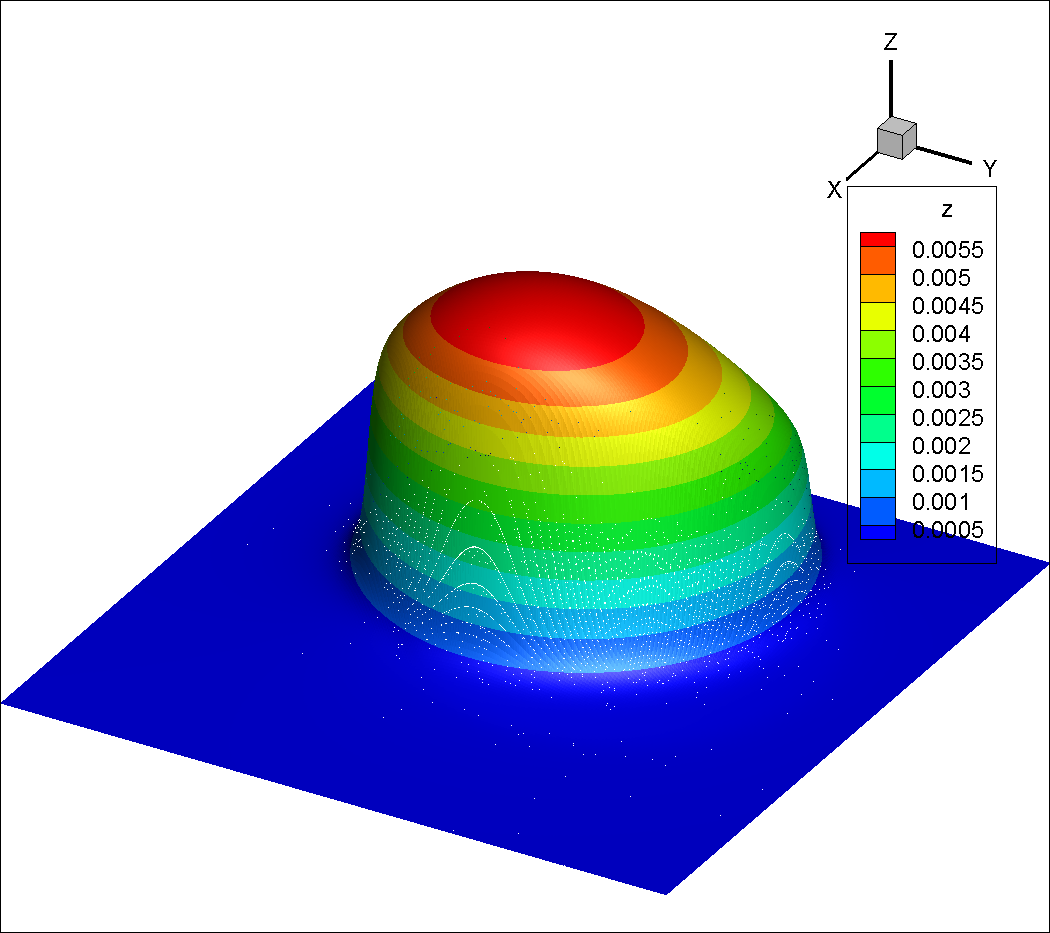}}
		}
		
		\caption{\label{fig:u1} Left is the exact solution $u_1$ and right is $u_{1h}$ computed by Ensemble HDG.}
	\end{figure}
	\begin{figure}
		\centerline{
			\hbox{\includegraphics[width=2.5in]{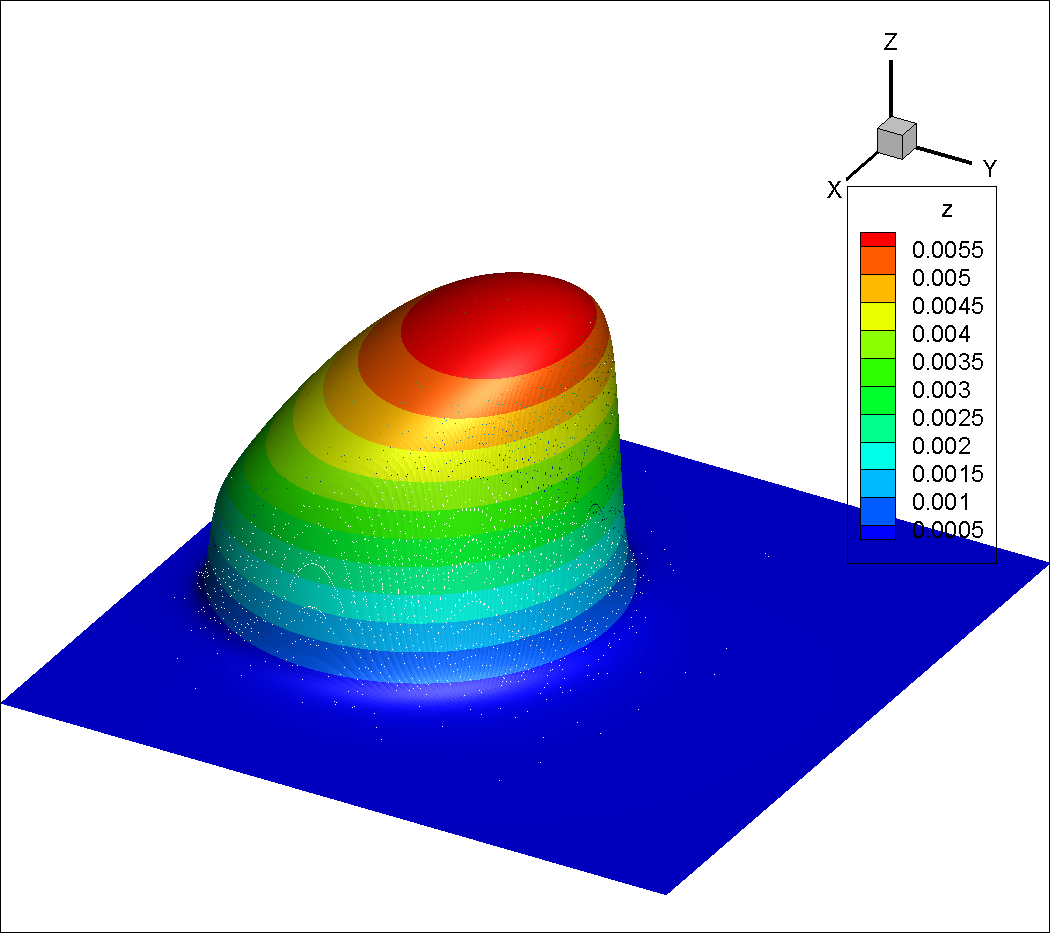}}
			\hbox{\includegraphics[width=2.5in]{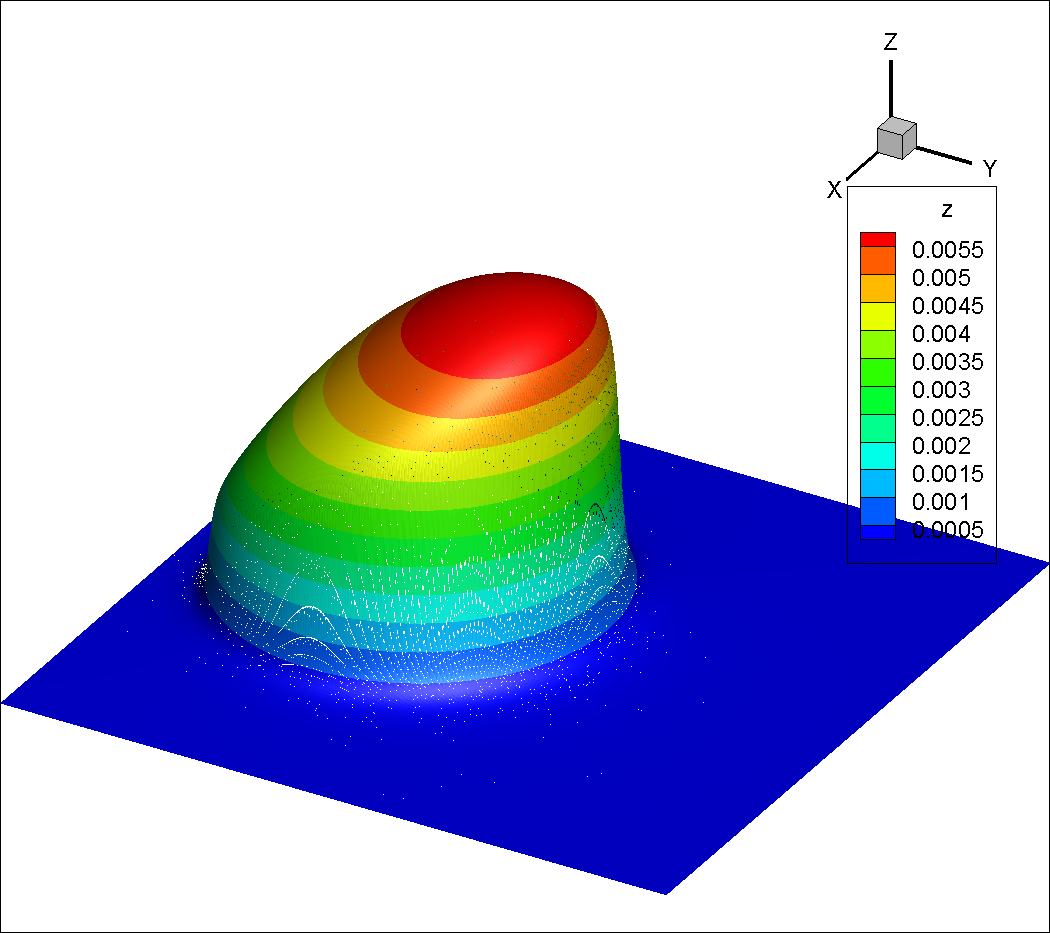}}
		}
		
		\caption{\label{fig:u2}  Left is the exact solution $u_2$ and right is $u_{2h}$ computed by Ensemble HDG.}
	\end{figure}
	\begin{figure}
		\centerline{
			\hbox{\includegraphics[width=2.5in]{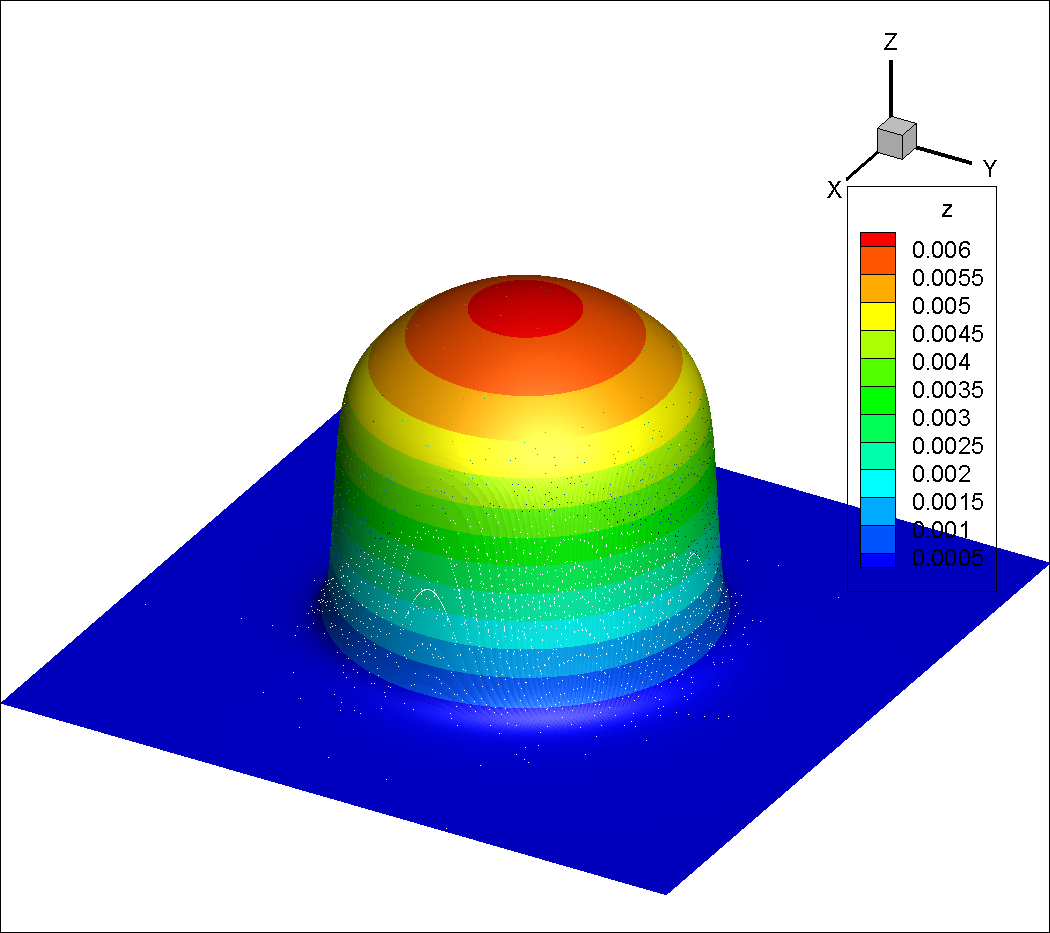}}
			\hbox{\includegraphics[width=2.5in]{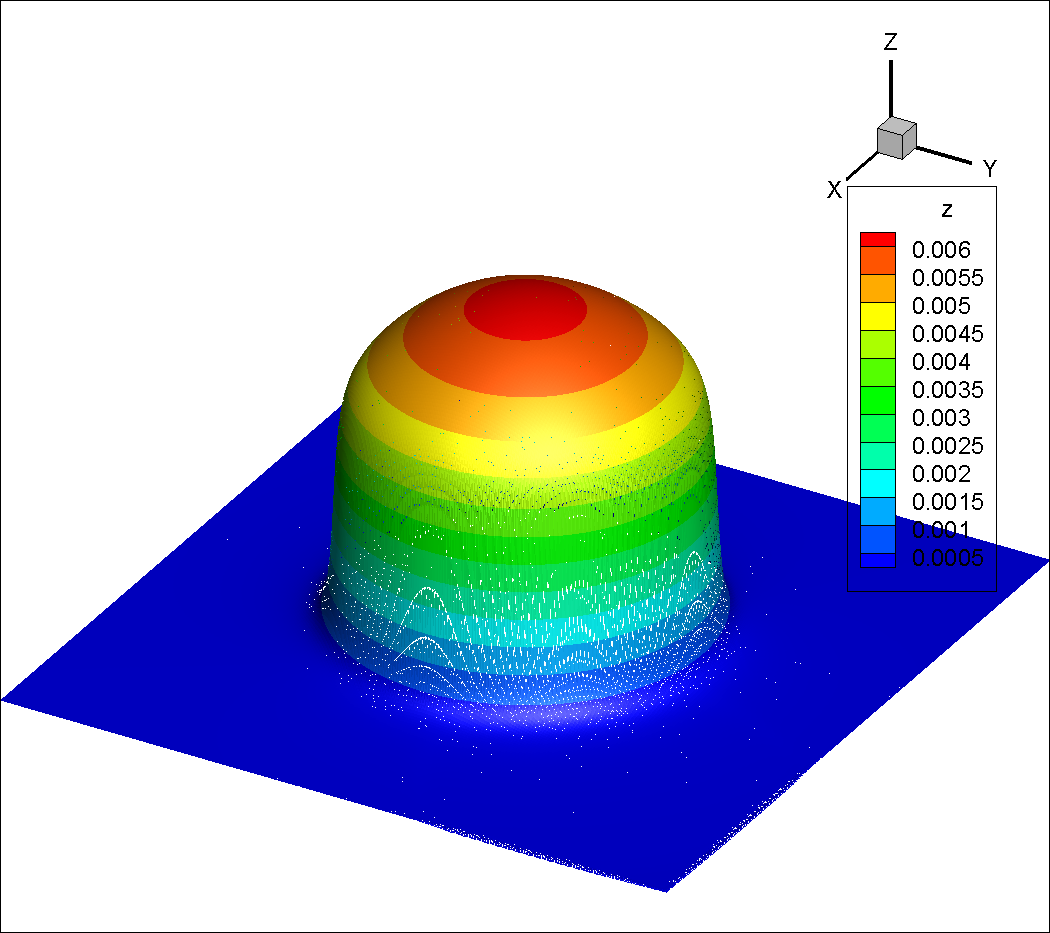}}
		}
		
		\caption{\label{fig:u3}  Left is the exact solution $u_3$ and right is $u_{3h}$ computed by Ensemble HDG.}
	\end{figure}
\end{example}

\begin{example}\label{example3}
	Fianlly, we perform the Ensemble HDG method for a group of  convection  dominated problems without exact solutions. In this example, the problems exhibit not interior layers but boundary layers. It is well known that the boundary layers are more difficult than interior layers for all numerical methods.  Since in \Cref{example2} the Ensemble HDG captured the interior layers without oscillations, we didn't plot the postprocessed solutions there. However,  our numerical test shows that the postprocessed solutions $u_{jh}^\star$ are better than $u_{jh}$ for solutions with  boundary layers;  see e.g. \Cref{fig:u4,fig:u5,fig:u6}. We note  there is no superconvergent rate even for a single convection dominated diffusion problem PDE using HDG methods, see,  e.g. \cite{Fu_Qiu_Zhang_Convection_Dominated_M2AN_2015}. 
	
	We plot all the ensemble members $u_{jh}$  and $u_{jh}^\star$  at the final time $T= 0.1$ for comparsion.  The domain, the mesh, the time step, the bounday conditions and the initial conditions are the same with \Cref{example2}. For the other data, we take 
	\begin{gather*}
	c_1 =60, \ c_2 =120, \ c_3 =180, \\ 
	\bm \beta_1 = [2,3], \ \bm \beta_2 = [3,4], \ \bm \beta_3 =[4,5], \\
	f_1 = 2, \  f_2 = 5, \ f_3 = 8.
	\end{gather*}
	\begin{figure}
		\centerline{
			\hbox{\includegraphics[width=2.5in]{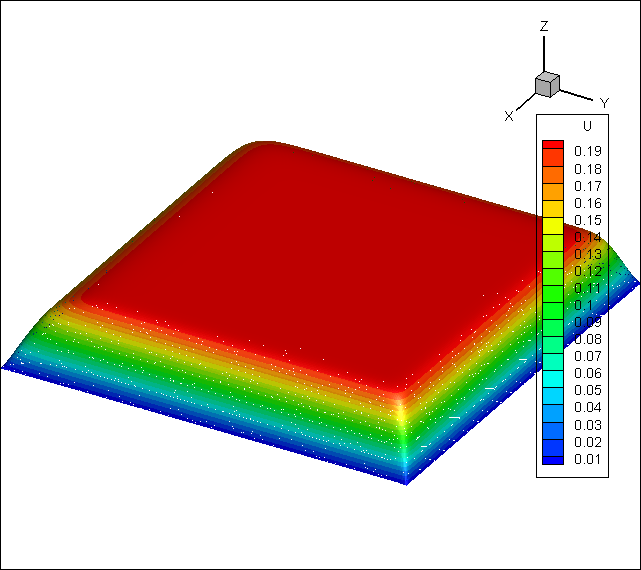}}
			\hbox{\includegraphics[width=2.5in]{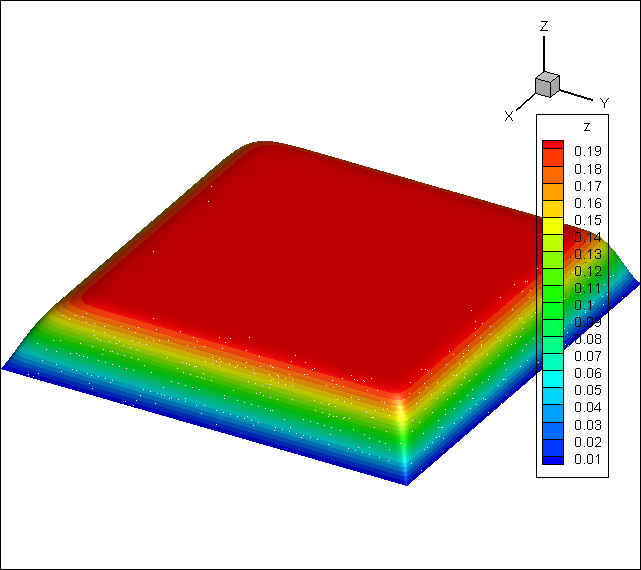}}
		}
		\caption{\label{fig:u4} Left is solution $u_{1h}$ and right is the postprocessed solution $u_{1h}^\star$.}
	\end{figure}
	\begin{figure}
		\centerline{
			\hbox{\includegraphics[width=2.5in]{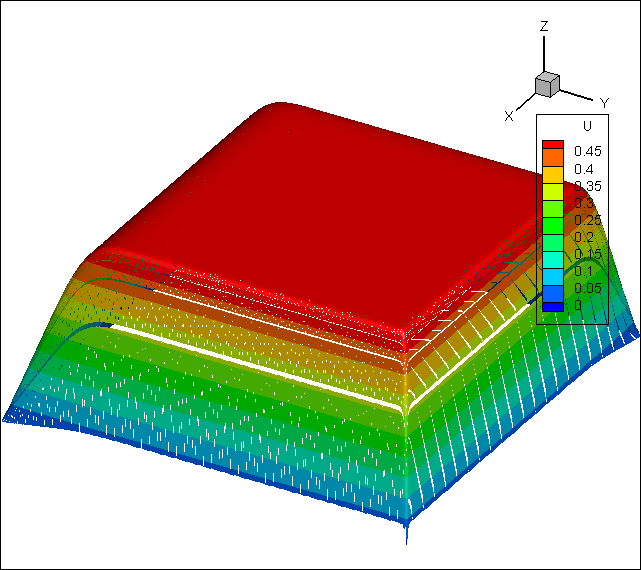}}
			\hbox{\includegraphics[width=2.5in]{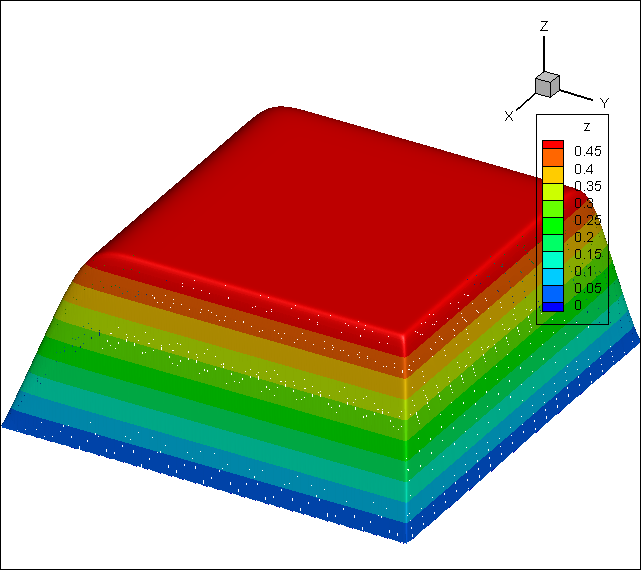}}
		}
		
		\caption{\label{fig:u5} Left is solution $u_{2h}$ and right is the postprocessed solution $u_{2h}^\star$.}
	\end{figure}
	\begin{figure}
		\centerline{
			\hbox{\includegraphics[width=2.5in]{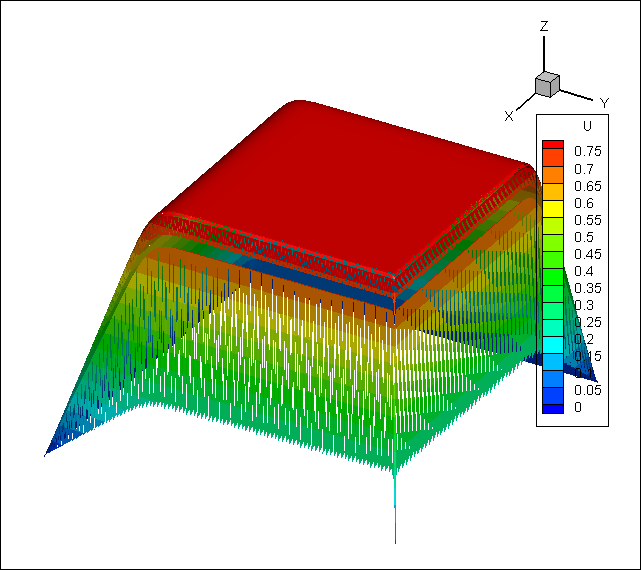}}
			\hbox{\includegraphics[width=2.5in]{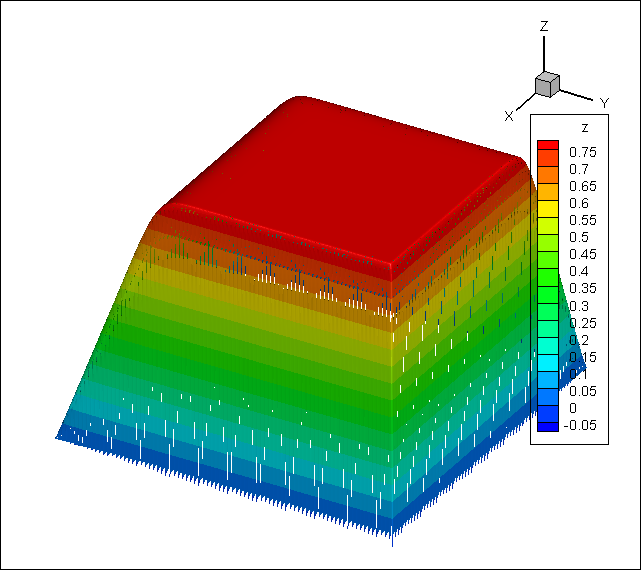}}
		}
		\caption{\label{fig:u6} Left is solution $u_{3h}$ and right is the postprocessed solution $u_{3h}^\star$.}
	\end{figure}
\end{example} 	

\section{Conclusion}
In this work, we first devised a superconvergent Ensemble HDG method for parameterized convection diffusion PDEs.  This Ensemble HDG method shares one common coefficient matrix and multiple RHS vectors, which is more efficient than performing separate simulations. We proved  optimal error estimates  for the flux $\bm q_j$ and the scalar variable $u_j$;  moreover, we obtained the superconvergent rate for $u_j$. As far as we are aware, this is the first time in the literature.

There are a number of topics that can be explored in the future, including devising  high order time stepping methods, a group of convection dominated diffusion PDEs,  and stochastic PDEs.

\section*{Acknowledgements}
G.\ Chen is supported by China Postdoctoral Science Fou-\\ndation grant 2018M633339. G.\ Chen thanks Missouri University of Science and Technology for hosting him as a visiting scholar; some of this work was completed during his research visit. L. Xu is supported in part by a Key Project of the Major Research Plan of NSFC grant no. 91630205 and the NSFC grant no. 11771068. The authors thank Dr. John Singler for his comments and edits, which significantly improved this paper.  
\section{Appendix}
\label{app}

In this section we only give a proof for 
\eqref{p1} and \eqref{p2}, since the rest are similar. To prove \eqref{error_dual_bb}-\eqref{error_dual_d}, we differentiate  the error equations in  \Cref{lemma41} with respect to time $t$.  It is easy to check that the operators $\Pi_W^j$  commute with the time derivative, i.e., $\partial_t \Pi_W^j u_j = \Pi_W^j \partial_tu_j $, since  the velocity vector fields $\bm \beta_j$ are independent of time $t$. 

\subsection{The equations of the projection of the errors}
\begin{lemma} \label{lemma41} We have the following equalities
	\begin{subequations}
		\begin{align*}
		(c_j  \bm{\Pi}_V^j\bm q_j,\bm{r}_h)_{\mathcal{T}_h}-(\Pi_W^ju_j,\nabla\cdot \bm{r}_h)_{\mathcal{T}_h}+\langle
		P_M u_j,\bm{r}_h\cdot \bm n \rangle_{\partial{\mathcal{T}_h}}
		= (c_j  (\bm\Pi_V^j\bm q_j-\bm q_j),\bm{r}_h)_{\mathcal{T}_h},\\
		(\nabla\cdot\bm\Pi_V^j \bm q_j, v_h)_{\mathcal{T}_h}-\langle \bm\Pi_V^j \bm q_j\cdot \bm{n},\widehat{v}_h\rangle_{\partial{\mathcal{T}_h}}
		+\langle \tau (\Pi_W^j u_j-P_M u_j),v_h-\widehat{v}_h \rangle_{\partial\mathcal{T}_h}\nonumber\\
		+( {\bm \beta}_j \cdot \nabla \Pi_W^j u_j, v_h)_{\mathcal{T}_h}
		-\langle{\bm \beta}_j \cdot\bm n, \Pi_W^j u_j \widehat{v}_h\rangle_{\partial\mathcal{T}_h}=(f_j-\partial_t u_j,v_h)_{\mathcal{T}_h},
		\end{align*}
		for all $(\bm r_h,v_h,\widehat{v}_h)\in \bm V_h\times W_h\times M_h$.
	\end{subequations}
\end{lemma}
The proof is similar to the proof of \Cref{lemma_error}, hence we omit it here.

To simplify the notation, we set
$$ \varepsilon^{u _j}_{h}=\overline u_{jh} -\Pi_W^j u_j ,\ \varepsilon^{\bm q_j}_{h}= \overline {\bm q}_{jh} -\bm \Pi_V ^j\bm q_j ,\
\varepsilon^{\widehat{u} _j}_{h}=\widehat{\overline u}_{jh} -P_M u_j.$$
Subtract  \eqref{projection} from \eqref{or} to get the following
\begin{lemma} \label{error_pi}We have the error equations
	\begin{subequations}\label{error}
		\begin{align}
		(c_j \varepsilon_{jh}^{\bm q },\bm{r}_h)_{\mathcal{T}_h}
		-(\varepsilon^{u }_{jh},\nabla\cdot \bm{r}_h)_{\mathcal{T}_h}+\langle \varepsilon^{\widehat{u} }_{jh},\bm{r}_h\cdot \bm n \rangle_{\partial{\mathcal{T}_h}}=(c_j (\bm\Pi_V^j\bm q_j-\bm q_j),\bm{r}_h)_{\mathcal{T}_h},\label{44a}\\
		(\nabla\cdot \varepsilon_{jh}^{\bm q }, v_h)_{\mathcal{T}_h}-\langle \varepsilon^{\bm{q} }_{jh}\cdot \bm{n},\widehat{v}_h\rangle_{\partial{\mathcal{T}_h}}
		+( {\bm \beta}_j \cdot \nabla \varepsilon^{u }_{jh}, v_h)_{\mathcal{T}_h}+	\langle \tau ( \varepsilon^{u }_{jh}-\varepsilon^{\widehat{u}}_{jh}),v_h-\widehat{v}_h \rangle_{\partial\mathcal{T}_h} \nonumber\\
		-\langle {\bm \beta}_j \cdot\bm n, \varepsilon^{{u} }_{jh} \widehat{v}_h\rangle_{\partial\mathcal{T}_h} =(\partial_t u_j-\partial_t \Pi_W^j u_j,v_h)_{\mathcal{T}_h}.\label{44b}
		\end{align}
	\end{subequations}
	for all $(\bm r_h,v_h,\widehat{v}_h)\in \bm V_h\times W_h\times M_h$.
\end{lemma}

\subsection{ Estimates for $ q_j$}
\begin{lemma}
	We have 
	\begin{align*}
	\hspace{1em}&\hspace{-1em}\|\sqrt{c_j}\varepsilon^{\bm{q} }_{jh}\|_{\mathcal{T}_h}
	+\|\sqrt{\tau}(\varepsilon^{u }_{jh}-\varepsilon^{{\widehat{u}} }_{jh})\|_{\partial\mathcal{T}_h}\\
	&\le  
	C \|\partial _t u_j-\partial _t \Pi_W^j u_j\|_{\mathcal{T}_h}+C \|\bm q_j-\bm\Pi_V^j \bm q_j\|_{\mathcal{T}_h} +C \|\varepsilon^{u }_{jh}\|_{\mathcal{T}_h}.
	\end{align*}
\end{lemma}
\begin{proof}
	We take $(\bm r_h,v_h,\widehat{v}_h)=(\varepsilon^{\bm{q} }_{jh},\varepsilon^{u }_{jh},\varepsilon^{{\widehat{u}}}_{jh} )$
	in \eqref{error}, and add them together to get
	\begin{align*}
	\hspace{1em}&\hspace{-1em} \|\sqrt{c_j}\varepsilon^{\bm{q} }_{jh}\|^2_{\mathcal{T}_h}
	+\|\sqrt{\tau}(\varepsilon^{u }_{jh}-\varepsilon^{{\widehat{u}} }_{jh})\|^2_{\mathcal{T}_h} + ( {\bm \beta}_j \cdot \nabla \varepsilon^{u }_{jh}, \varepsilon^{u }_{jh})_{\mathcal{T}_h}
	-\langle {\bm \beta}_j \cdot\bm n, \varepsilon^{{u} }_{jh} \varepsilon^{\widehat{u} }_{jh}\rangle_{\partial\mathcal{T}_h}\nonumber\\
	&=
	(c_j  (\bm\Pi_V^j\bm q_j-\bm q_j),\varepsilon^{\bm q }_{jh})_{\mathcal{T}_h}+
	(\partial_t u_j-\partial_t \Pi_W^j u_j,\varepsilon^{u }_{jh})_{\mathcal{T}_h}.
	\end{align*}
	By Green's formula and the fact
	$\langle ({\bm \beta}_j \cdot\bm n )\varepsilon^{\widehat{u} }_{jh},  \varepsilon^{\widehat{u} }_{jh}\rangle_{\partial\mathcal{T}_h}=0$
	we have
	\begin{align}
	( {\bm \beta}_j \cdot \nabla \varepsilon^{u }_{jh}, \varepsilon^{u }_{jh})_{\mathcal{T}_h}
	-\langle {\bm \beta}_j \cdot\bm n, \varepsilon^{{u} }_{jh} \varepsilon^{\widehat{u} }_{jh}\rangle_{\partial\mathcal{T}_h}
	\le \frac{1}{2}\|\sqrt{ |{\bm \beta}_j \cdot\bm n}|
	(\varepsilon^{u }_{jh}-\varepsilon^{\widehat{u} }_{jh})\|^2_{\partial\mathcal{T}_h}.
	\end{align}
	Then by  condition \eqref{tau}, we get the desired result.
\end{proof}

\subsection{Dual arguments}

The next step is the consideration of the dual problems:
\begin{equation}\label{Dual_PDE1}
\begin{split}
c_j\bm{\Phi}_j+\nabla\Psi_j&=0\qquad\qquad~\text{in}\ \Omega,\\
\nabla\cdot\bm \Phi_j-\bm{\beta}_j\cdot\nabla\Psi_j  &=\Theta_j\qquad\quad~~\text{in}\ \Omega,\\
\Psi_j &= 0\qquad\qquad~\text{on}\ \partial\Omega.
\end{split}
\end{equation}
{\bfseries{Elliptic regularity}.} To obatin the superconvergent rate, we are going to assume that the domain $\Omega$ is such that for any $\Theta_j \in L^2(\Omega)$, we have the regularity estimates
for  these boundary value problems \eqref{Dual_PDE1}:
\begin{align}\label{regularity_PDE}
\|\bm \Phi_j\|_{H^{1}(\Omega)} + \|\Psi_j\|_{H^{2}(\Omega)} \le C   \|\Theta_j\|_{L^{2}(\Omega)}.
\end{align}
It is well known that this holds whenever $\Omega$ is a convex polyhedral domain.

\begin{lemma} 
	If the elliptic regularity inequality \eqref{Dual_PDE1} holds, then we have the error estimates
	\begin{align*}
	\|\sqrt{ c_j }\varepsilon^{\bm{q} }_{jh}\|_{\mathcal{T}_h}
	+\|\sqrt{\tau}(\varepsilon^{u}_{jh}-\varepsilon^{{\widehat{u}} }_{jh})\|_{\partial \mathcal{T}_h} \le C\mathcal A_j,\\
	\|\varepsilon^{u }_{jh}\|_{\mathcal{T}_h} \le Ch^{k+\min\{k,1\}} \mathcal A_j,
	\end{align*}
	where
	\begin{align*}
	\mathcal A_j = \|u_j - \Pi_W^j u_j\|_{\mathcal T_h}+\|\bm q_j - \bm\Pi_V^j \bm q_j\|_{\mathcal T_h} +\|\partial_tu_j - \Pi_W^j \partial_tu_j\|_{\mathcal T_h}.
	\end{align*}
\end{lemma}
\begin{proof} Similar to \Cref{lemma41}, we have the following equations:
	\begin{align*}
	(c_j \bm\Pi_V^j\bm \Phi_j, \bm{r}_h)_{\mathcal{T}_h}-(\Pi_W^j\Psi,\nabla\cdot \bm{r}_h)_{\mathcal{T}_h}+\langle
	P_M \Psi_j,\bm{r}_h\cdot \bm n \rangle_{\partial{\mathcal{T}_h}}
	=
	(c_j  (\bm \Pi_V^j\bm \Phi_j-\bm \Phi_j),\bm{r}_h)_{\mathcal{T}_h},\\
	(\nabla\cdot\bm\Pi_V^j \bm \Phi_j, v_h)_{\mathcal{T}_h}-\langle \bm\Pi_V^j \bm  \Phi_j \cdot \bm{n},\widehat{v}_h\rangle_{\partial{\mathcal{T}_h}}
	+\langle \tau (\Pi_W^j \Psi-P_M \Psi_j),v_h-\widehat{v}_h \rangle_{\partial\mathcal{T}_h}\nonumber\\
	-( {\bm \beta}_j \cdot \nabla \Pi_W^j \Psi_j, v_h)_{\mathcal{T}_h}
	+\langle{\bm \beta}_j \cdot\bm n, \Pi_W^j u_j \widehat{v}_h\rangle_{\partial\mathcal{T}_h}=({\Theta}_j,v_h)_{\mathcal{T}_h}.
	\end{align*}
	Take $(\bm r_h,v_h,\widehat{v}_h)=(\varepsilon^{\bm{q} }_{jh},\varepsilon^{u }_{jh},\varepsilon^{{\widehat{u}}}_{jh} )$ and $\Theta_j=\varepsilon^{u}_{jh}$  above to get,  
	\begin{align*}
	\|\varepsilon^{u }_{jh}\|^2_{\mathcal{T}_h}
	&=(\nabla\cdot\bm\Pi_V^j \bm\Phi_j, \varepsilon^{u }_{jh})_{\mathcal{T}_h}-\langle \bm\Pi_V^j \bm \Phi_j\cdot \bm{n},\varepsilon^{{\widehat{u}} }_{jh} \rangle_{\partial{\mathcal{T}_h}}\\
	&\quad+\langle \tau (\Pi_W^j \Psi_j-P_M \Psi_j),\varepsilon^{u }_{jh}-\varepsilon^{{\widehat{u}} }_{jh} \rangle_{\partial\mathcal{T}_h}\\
	&\quad -( {\bm \beta}_j \cdot \nabla \Pi_W^j \Psi_j, \varepsilon^{u }_{jh})_{\mathcal{T}_h}
	+\langle{\bm \beta}_j \cdot\bm n, \Pi_W^j\Psi_j \varepsilon^{{\widehat{u}} }_{jh} \rangle_{\partial\mathcal{T}_h}.
	\end{align*}
	By \eqref{44a} one gets
	\begin{align*}
	\|\varepsilon^{u }_{jh}\|^2_{\mathcal{T}_h}
	&=(c_j \varepsilon_{jh}^{\bm q },\bm\Pi_V^j \bm\Phi_j)_{\mathcal{T}_h}
	-(c_j  (\bm\Pi_V^j\bm q_j-\bm q_j),\bm\Pi_V^j \bm\Phi_j)_{\mathcal{T}_h}+\langle{\bm \beta}_j \cdot\bm n, \Pi_W^j \Psi_j \varepsilon^{{\widehat{u}} }_{jh} \rangle_{\partial\mathcal{T}_h}\\
	&\quad +\langle \tau (\Pi_W^j \Psi_j-P_M \Psi_j),\varepsilon^{u }_{jh}-\varepsilon^{{\widehat{u}} }_{jh}\rangle_{\partial\mathcal{T}_h} -({\bm \beta}_j \cdot \nabla \Pi_W^j \Psi_j, \varepsilon^{u }_{jh})_{\mathcal{T}_h}.
	\end{align*}
	Hence, 
	\begin{align*}
	\|\varepsilon^{u }_{jh}\|^2_{\mathcal{T}_h}
	&=(\Pi_W^j\Psi_j,\nabla\cdot \varepsilon_{jh}^{\bm q })_{\mathcal{T}_h}-\langle
	P_M \Psi_j,\varepsilon_{jh}^{\bm q }\cdot \bm n \rangle_{\partial{\mathcal{T}_h}}-
	(c_j  (\bm\Pi_V^j\bm \Phi_j-\bm \Phi_j),\varepsilon_{jh}^{\bm q })_{\mathcal{T}_h}\\
	&\quad-(c_j  (\bm\Pi_V^j\bm q_j-\bm q_j),\bm\Pi_V^j \bm \Phi_j)_{\mathcal{T}_h}
	+\langle \tau (\Pi_W^j \Psi_j-P_M \Psi_j),\varepsilon^{u }_{jh}-\varepsilon^{{\widehat{u}} }_{jh}  \rangle_{\partial\mathcal{T}_h}\\
	&\quad-( {\bm \beta}_j \cdot \nabla \Pi_W^j \Psi_j, \varepsilon^{u }_{jh})_{\mathcal{T}_h}
	+\langle{\bm \beta}_j \cdot\bm n, \Pi_W^j \Psi_j \varepsilon^{{\widehat{u}} }_{jh} \rangle_{\partial\mathcal{T}_h}.
	\end{align*}
	By Green's formula one gets
	\begin{align*}
	\|\varepsilon^{u }_{jh}\|^2_{\mathcal{T}_h}
	&=(\Pi_W^j\Psi,\nabla\cdot \varepsilon_{jh}^{\bm q })_{\mathcal{T}_h}-\langle
	P_M \Psi_j,\varepsilon_{jh}^{\bm q }\cdot \bm n \rangle_{\partial{\mathcal{T}_h}}-
	(c_j  (\bm\Pi_V^j\bm \Phi_j-\bm \Phi_j),\varepsilon_{jh}^{\bm q })_{\mathcal{T}_h}\\
	&\quad -(c_j  (\bm\Pi_V^j\bm q_j-\bm q_j),\bm\Pi_V^j \bm \Phi_j)_{\mathcal{T}_h}
	+\langle \tau (\Pi_W^j \Psi-P_M \Psi),\varepsilon^{u }_{jh}-\varepsilon^{{\widehat{u}} }_{jh}  \rangle_{\partial\mathcal{T}_h}\\
	&\quad +( {\bm \beta}_j \cdot \nabla \varepsilon^{u }_{jh}, \Pi_W^j \Psi_j)_{\mathcal{T}_h}
	+\langle{\bm \beta}_j \cdot\bm n, \Pi_W^j \Psi_j (\varepsilon^{{\widehat{u}} }_{jh}-\varepsilon^{u }_{jh} ) \rangle_{\partial\mathcal{T}_h}\\
	&=(\Pi_W^j\Psi_j,\nabla\cdot \varepsilon_{jh}^{\bm q})_{\mathcal{T}_h}-\langle
	P_M \Psi_j,\varepsilon_{jh}^{\bm q }\cdot \bm n \rangle_{\partial{\mathcal{T}_h}}-
	(c_j  (\bm\Pi_V^j\bm \Phi_j-\bm \Phi_j),\varepsilon_{jh}^{\bm q })_{\mathcal{T}_h}\\
	&\quad -(c_j  (\bm\Pi_V^j\bm q_j-\bm q_j),\bm\Pi_V^j \bm \Phi_j)_{\mathcal{T}_h}
	+\langle \tau (\Pi_W^j \Psi_j-P_M \Psi_j),\varepsilon^{u }_{jh}-\varepsilon^{{\widehat{u}} }_{jh} \rangle_{\partial\mathcal{T}_h}\\
	&\quad+( {\bm \beta}_j \cdot \nabla \varepsilon^{u }_{jh}, \Pi_W^j \Psi_j)_{\mathcal{T}_h}
	-\langle {\bm \beta}_j \cdot\bm n, \varepsilon^{{u} }_{jh} P_M\Psi_j\rangle_{\partial\mathcal{T}_h}+\langle {\bm \beta}_j \cdot\bm n, \varepsilon^{{u} }_{jh} P_M\Psi_j\rangle_{\partial\mathcal{T}_h}\\
	&\quad +\langle{\bm \beta}_j \cdot\bm n, \Pi_W^j \Psi_j (\varepsilon^{{\widehat{u}} }_{jh}-\varepsilon^{u }_{jh}) \rangle_{\partial\mathcal{T}_h}.
	\end{align*}
	By \eqref{44b} one gets
	\begin{align*}
	\|\varepsilon^{u }_{jh}\|^2_{\mathcal{T}_h}
	&=-
	(c_j  (\bm \Pi_V^j\bm \Phi_j-\bm \Phi_j),\varepsilon_{jh}^{\bm q })_{\mathcal{T}_h}
	-(c_j  (\bm\Pi_V^j\bm q_j-\bm q_j),\bm\Pi_V^j \bm \Phi_j)_{\mathcal{T}_h}\\
	&\quad+\langle {\bm \beta}_j \cdot\bm n, (\varepsilon^{{u} }_{jh}-\varepsilon_{jh}^{\widehat{u} }) (P_M\Psi_j-\Pi_W^j\Psi_j)\rangle_{\partial\mathcal{T}_h}\\
	&\quad+( {\bm \beta}_j \cdot \nabla  \Pi_W^j\Psi_j, \Pi_W^ju_j-u_j)_{\mathcal{T}_h} +(\partial_t u_j-\partial_t \Pi_W^j u_j,\Pi_W^j\Psi_j)_{\mathcal T_h}\\
	&=\sum_{i=1}^5 R_i.
	\end{align*}
	We estimate $\{R_i\}_{i=1}^5$ term by term:
	\begin{align*}
	R_1&\le Ch\|\bm{\Phi}_j\|_{1}
	\|\varepsilon_{jh}^{\bm q }\|_{\mathcal{T}_h}\\
	& \le C h\|\varepsilon_{jh}^{ u }\|_{\mathcal{T}_h}^2
	+Ch(\|u_j-\Pi_W^ju_j\|_{\mathcal{T}_h}
	+\|\bm q_j-\bm\Pi_V^j \bm q_j\|_{\mathcal{T}_h})
	\|\varepsilon_{jh}^{ u }\|_{\mathcal{T}_h},\\
	R_2&\le Ch^{\min\{k,1\}}\|\bm{\Phi}_j\|_{1}\|\bm q_j-\bm\Pi_V^j \bm q_j\|_{\mathcal{T}_h} \le Ch^{\min\{k,1\}}\|\varepsilon_{jh}^{u}\|_{\mathcal{T}_h}\|\bm q_j-\bm\Pi_V^j \bm q_j\|_{\mathcal{T}_h},\\
	R_3&\le Ch^{\frac 1 2 +\min\{k,1\}}\|\Psi_j\|_2
	\|\sqrt{\tau}(\varepsilon_{jh}^{u}-\varepsilon_{jh}^{\widehat{u} })\|_{\partial\mathcal{T}_h}\\
	& \le Ch^{\frac 1 2+\min\{k,1\}}
	(\|u_j-\Pi_W^ju_j\|_{\mathcal{T}_h}
	+\|\bm q_j-\bm\Pi_V^j \bm q_j\|_{\mathcal{T}_h})
	\|\varepsilon_{jh}^{ u }\|_{\mathcal{T}_h}\\
	&\quad+ Ch^{\frac 1 2+\min\{k,1\}}\|\varepsilon_{jh}^{ u }\|_{\mathcal{T}_h}^2,\\
	R_4&=( ({\bm \beta}_j -\bm{\Pi}_0\bm\beta_j )\cdot \nabla  \Pi_W^j \Psi_j, \Pi_W^ju_j-u_j)_{\mathcal{T}_h}\\
	&\le Ch\|\Psi_j\|_1\|u_j-\Pi_W^ju_j\|_{\mathcal{T}_h} \le Ch\|\varepsilon_{jh}^{u }\|_{\mathcal{T}_h}\|u_j-\Pi_W^ju_j\|_{\mathcal{T}_h},\\
	R_5&\le Ch^{\min\{k,1\}}\|\Psi_j\|_1
	\|\partial_tu_j-\Pi_W^j\partial_tu_j\|_{\mathcal{T}_h} \le Ch^{\min\{k,1\}}\|\varepsilon^{u }_{jh}\|_{\mathcal{T}_h} \|\partial_tu_j-\Pi_W^j\partial_tu_j\|_{\mathcal{T}_h}.
	\end{align*}
	Hence, we have 
	\begin{align*}
	\|\varepsilon^{u }_{jh}\|_{\mathcal{T}_h} \le Ch^{\min\{k,1\}}
	\left(\|u_j-\Pi_W^ju_j\|_{\mathcal{T}_h}
	+\|\bm q_j-\bm\Pi_V^j \bm q_j\|_{\mathcal{T}_h}
	+ \|\partial_tu_j-\Pi_W^j\partial_tu_j\|_{\mathcal{T}_h}
	\right).
	\end{align*}
\end{proof}

\bibliographystyle{plain}

\bibliography{Model_Order_Reduction,Ensemble,HDG,Interpolatory,Mypapers,Added,Drift_Diffusion}

\end{document}